\documentclass[reqno]{amsart}
\usepackage[latin1]{inputenc}
\usepackage[pdftex]{graphicx}
\usepackage{graphics}
\usepackage{graphicx}
\usepackage{epstopdf}
\usepackage[pdftex,hyperindex]{hyperref}
\usepackage{floatflt}
\usepackage{multicol}
\usepackage[T1]{fontenc}
\usepackage{textcomp}
\usepackage{latexsym}
\usepackage{expdlist}
\usepackage{vmargin}
\usepackage{booktabs}
\usepackage{placeins}
\usepackage{array}
\usepackage{amsmath}
\usepackage{amssymb}
\usepackage{amsfonts}
\usepackage{verbatim}
\usepackage{array}
\usepackage{framed}
\usepackage{amsthm}
\usepackage{enumerate}
\usepackage{cite}
\input amssym.def
\input amssym
\theoremstyle{plain}
\newtheorem{theorem}{Theorem}[section]
\newtheorem{lemma}{Lemma}[section]
\newtheorem{corollary}{Corollary}[section]
\newtheorem{proposition}{Proposition}[section]
\theoremstyle{remark}
\newtheorem{remark}{Remark}[section]

\theoremstyle{definition}

\allowdisplaybreaks[1]

\begin{document}

\title[Jacobi and generalized Chebyshev polynomials]{Nonnegative and strictly positive linearization of Jacobi and generalized Chebyshev polynomials}

\author{Stefan Kahler}

\address{Department of Mathematics, Chair for Mathematical Modelling, Chair for Mathematical Modeling of Biological Systems, Technical University of Munich, Boltzmannstr. 3, 85747 Garching b. M\"{u}nchen, Germany}

\address{Lehrstuhl A f\"{u}r Mathematik, RWTH Aachen University, 52056 Aachen, Germany}

\email{stefan.kahler@mathA.rwth-aachen.de}

\thanks{The research was begun when the author worked at Technical University of Munich, and the author gratefully acknowledges support from the graduate program TopMath of the ENB (Elite Network of Bavaria) and the TopMath Graduate Center of TUM Graduate School at Technical University of Munich. The main part of the research was done at RWTH Aachen University.}

\date{\today}

\begin{abstract}
In the theory of orthogonal polynomials, as well as in its intersection with harmonic analysis, it is an important problem to decide whether a given orthogonal polynomial sequence $(P_n(x))_{n\in\mathbb{N}_0}$ satisfies nonnegative linearization of products, i.e., the product of any two $P_m(x),P_n(x)$ is a conical combination of the polynomials $P_{|m-n|}(x),\ldots,P_{m+n}(x)$. Since the coefficients in the arising expansions are often of cumbersome structure or not explicitly available, such considerations are generally very nontrivial. In 1970, G. Gasper was able to determine the set $V$ of all pairs $(\alpha,\beta)\in(-1,\infty)^2$ for which the corresponding Jacobi polynomials $(R_n^{(\alpha,\beta)}(x))_{n\in\mathbb{N}_0}$, normalized by $R_n^{(\alpha,\beta)}(1)\equiv1$, satisfy nonnegative linearization of products. In 2005, R. Szwarc asked to solve the analogous problem for the generalized Chebyshev polynomials $(T_n^{(\alpha,\beta)}(x))_{n\in\mathbb{N}_0}$, which are the quadratic transformations of the Jacobi polynomials and orthogonal w.r.t. the measure $(1-x^2)^{\alpha}|x|^{2\beta+1}\chi_{(-1,1)}(x)\,\mathrm{d}x$. In this paper, we give the solution and show that $(T_n^{(\alpha,\beta)}(x))_{n\in\mathbb{N}_0}$ satisfies nonnegative linearization of products if and only if $(\alpha,\beta)\in V$, so the generalized Chebyshev polynomials share this property with the Jacobi polynomials. Moreover, we reconsider the Jacobi polynomials themselves, simplify Gasper's original proof and characterize strict positivity of the linearization coefficients. Our results can also be regarded as sharpenings of Gasper's one.
\end{abstract}



\maketitle

\numberwithin{equation}{section}

\section{Introduction}\label{sec:intro}

\subsection{Motivation}

In the theory of orthogonal polynomials and special functions, it is of special interest under which conditions (general or referring to a specific class of polynomials) a suitably normalized orthogonal polynomial sequence $(P_n(x))_{n\in\mathbb{N}_0}\subseteq\mathbb{R}[x]$ satisfies the `nonnegative linearization of products' property, i.e., the product of any two polynomials $P_m(x),P_n(x)$ is contained in the conical hull of $\{P_k(x):k\in\mathbb{N}_0\}$. In other words, nonnegative linearization of products means that the linearization coefficients appearing in the (Fourier) expansions of $P_m(x)P_n(x)$ w.r.t. the basis $\{P_k(x):k\in\mathbb{N}_0\}$ are always nonnegative. One reason for the intense study of this property, and for the extensive corresponding literature, is a fruitful relation to harmonic analysis, which will be briefly recalled below.\\

Given a specific sequence $(P_n(x))_{n\in\mathbb{N}_0}$, deciding whether nonnegative linearization of products is satisfied or not may be quite difficult, however: in many cases, the aforementioned linearization coefficients are not explicitly known, or explicit representations are of involved, cumbersome or inappropriate structure. In a series of papers starting with \cite{Sz92b} and extending earlier work of Askey \cite{As70}, Szwarc et al. have provided some general criteria that can be helpful. However, to our knowledge there is no general criterion which is strong enough to cover the full parameter range for which the Jacobi polynomials
\begin{equation*}
R_n^{(\alpha,\beta)}(x)={}_2F_1\left(\left.\begin{matrix}-n,n+\alpha+\beta+1 \\ \alpha+1\end{matrix}\right|\frac{1-x}{2}\right)=\sum_{k=0}^n\frac{(-n)_k(n+\alpha+\beta+1)_k}{(\alpha+1)_k}\frac{(1-x)^k}{2^k k!}
\end{equation*}
\cite[(9.8.1)]{KLS10} satisfy nonnegative linearization of products.\footnote{Recall that $(a)_0=1$ and $(a)_n=\prod_{k=1}^n(a+k-1)\;(n\in\mathbb{N})$. Since $(R_n^{(\alpha,\beta)}(x))_{n\in\mathbb{N}_0}$ is normalized such that $R_n^{(\alpha,\beta)}(1)\equiv1$, one has $R_n^{(\alpha,\beta)}(x)=n!P_n^{(\alpha,\beta)}(x)/(\alpha+1)_n$ if $(P_n^{(\alpha,\beta)}(x))_{n\in\mathbb{N}_0}$ denotes the standard normalization of the Jacobi polynomials.} Moreover, we are not aware of an explicit representation of the corresponding linearization coefficients which allows one to easily identify \textit{all} pairs $(\alpha,\beta)\in(-1,\infty)^2$ such that nonnegative linearization of products is fulfilled.\\

In some more detail, the situation concerning Jacobi polynomials is as follows: starting with the full (positive-definite case) parameter range $(\alpha,\beta)\in(-1,\infty)^2$ and defining
\begin{equation}\label{eq:abdef}
a:=\alpha+\beta+1>-1,\;b:=\alpha-\beta\in(-1-a,1+a)
\end{equation}
and a proper subset $V$ of $[-1/2,\infty)\times(-1,\infty)$ via
\begin{equation}\label{eq:Voriginal}
V:=\left\{(\alpha,\beta)\in(-1,\infty)^2:a(a+5)(a+3)^2\geq(a^2-7a-24)b^2,b\geq0\right\},
\end{equation}
Gasper showed the following \cite[Theorem 1]{Ga70b} (or \cite[Theorem 3]{Ga75}):

\begin{theorem}\label{thm:gasper}
Let $\alpha,\beta>-1$. The following are equivalent:
\begin{enumerate}
\item[\rm(i)] $(R_n^{(\alpha,\beta)}(x))_{n\in\mathbb{N}_0}$ satisfies nonnegative linearization of products, i.e., all $g_R(m,n;k)$ given by the expansions
\begin{equation*}
R_m^{(\alpha,\beta)}(x)R_n^{(\alpha,\beta)}(x)=\sum_{k=|m-n|}^{m+n}g_R(m,n;k)R_k^{(\alpha,\beta)}(x)
\end{equation*}
are nonnegative.
\item[\rm(ii)] $(\alpha,\beta)\in V$.
\end{enumerate}
\end{theorem}

Although Theorem~\ref{thm:gasper} can be regarded as a ``classical result'' nowadays, it is still highly topical and was used in the recent publications \cite{BGM18} (dealing with certain strictly positive definite functions) and \cite{GM18} (dealing with semigroups defined by Fourier-Jacobi series), for instance. Also \cite{BDHSS19}, dealing with spherical codes, uses \cite{Ga70b}.\\

In this paper, we find an analogue to Gasper's result Theorem~\ref{thm:gasper} for the class of generalized Chebyshev polynomials (Theorem~\ref{thm:mainfull}), which are the quadratic transformations of the Jacobi polynomials; for all $\alpha,\beta>-1$, the sequence of generalized Chebyshev polynomials $(T_n^{(\alpha,\beta)}(x))_{n\in\mathbb{N}_0}$ is given by
\begin{align}
\label{eq:genchebeven} T_{2n}^{(\alpha,\beta)}(x)&:=R_n^{(\alpha,\beta)}(2x^2-1),\\
\label{eq:genchebodd} T_{2n+1}^{(\alpha,\beta)}(x)&:=x R_n^{(\alpha,\beta+1)}(2x^2-1)
\end{align}
\cite[Chapter V 2 (G)]{Ch78}. This solves a problem posted in \cite[Section 5]{Sz05} by Szwarc who asked to determine the parameter range for which these polynomials satisfy nonnegative linearization of products. Since our result will immediately imply the nontrivial direction (cf. below) of Gasper's result Theorem~\ref{thm:gasper}, it can also be regarded as a sharpening of Theorem~\ref{thm:gasper}.\\

Moreover, we shall characterize strict positivity\footnote{In the following, we just write `positive' for `strictly positive' etc. If $0$ shall be included, we use `nonnegative'.} of the linearization coefficients $g_R(m,n;k)$ (Theorem~\ref{thm:gasperpositivevariant}); an analogous result for the generalized Chebyshev polynomials cannot exist due to symmetry. On the one hand, this characterization will immediately imply the nontrivial direction of Gasper's result Theorem~\ref{thm:gasper}, too. On the other hand, our proof of positive linearization is based on Gasper's approach \cite{Ga70b} but shorter and more elementary, so it can be regarded both as another sharpening and as a simplification. In Gasper's original proof, the most computational part was establishing the nonnegativity of the coefficients $g_R(m,n;|m-n|+2)$ and $g_R(m,n;m+n-2)$ (provided $(\alpha,\beta)\in V$). We will get rid of these long computations and provide a more explicit approach. Furthermore, we give characterizations concerning a certain oscillatory behavior of the $g_R(m,n;k)$.

\subsection{Underlying setting}

Let us briefly describe the basic underlying setting: in this paper, we consider sequences $(P_n(x))_{n\in\mathbb{N}_0}\subseteq\mathbb{R}[x]$ with $\mathrm{deg}\;P_n(x)=n$ which are orthogonal w.r.t. a probability (Borel) measure $\mu$ on the real line with $|\mathrm{supp}\;\mu|=\infty$ and $\mathrm{supp}\;\mu\subseteq(-\infty,1]$. Moreover, we assume $(P_n(x))_{n\in\mathbb{N}_0}$ to be normalized by $P_n(1)\equiv1$. This normalization is always possible because the assumptions on $\mathrm{supp}\;\mu$ yield that all zeros are (real and) located in $(-\infty,1)$ (see \cite{Ch78,La05} for standard results on orthogonal polynomials and on corresponding expansions). Orthogonality is then given by
\begin{equation}\label{eq:orthogonalityrelation}
\int_{\mathbb{R}}\!P_m(x)P_n(x)\,\mathrm{d}\mu(x)=\frac{\delta_{m,n}}{h(n)}
\end{equation}
with some function $h:\mathbb{N}_0\rightarrow(0,\infty)$ satisfying $h(0)=1$.\\

Under these conditions, nonnegative linearization of products corresponds to the property that the product of any two polynomials $P_m(x),P_n(x)$ is a convex combination of $P_{|m-n|}(x),\ldots,P_{m+n}(x)$, or to the nonnegativity of all linearization coefficients $g(m,n;k)$ defined by the expansions
\begin{equation}\label{eq:productlinear}
P_m(x)P_n(x)=\sum_{k=|m-n|}^{m+n}g(m,n;k)P_k(x),
\end{equation}
where $\sum_{k=|m-n|}^{m+n}g(m,n;k)=1$.\\

Observe that the summation in \eqref{eq:productlinear} starts with $k=|m-n|$, which is an obvious consequence of orthogonality. It is also easy to see that $g(m,n;m+n)>0$ (as the assumptions on the support and the normalization yield that the polynomials $P_n(x)$ have positive leading coefficients). Using \eqref{eq:orthogonalityrelation} and \eqref{eq:productlinear}, one clearly has
\begin{equation}\label{eq:lincoeffh}
g(m,n;k)=h(k)\int_{\mathbb{R}}\!P_m(x)P_n(x)P_k(x)\,\mathrm{d}\mu(x)
\end{equation}
and $h(n)=1/g(n,n;0)$. As an immediate consequence of \eqref{eq:lincoeffh} and the observation $g(m,n;m+n)>0$ above, we obtain that also $g(m,n;|m-n|)>0$. In particular, it is well-known that $(P_n(x))_{n\in\mathbb{N}_0}$ satisfies the three-term recurrence relation $P_0(x)=1$, $P_1(x)=(x-b_0)/a_0$,
\begin{equation*}
P_1(x)P_n(x)=a_n P_{n+1}(x)+b_n P_n(x)+c_n P_{n-1}(x)\;(n\in\mathbb{N}),
\end{equation*}
where $a_0>0$, $b_0=1-a_0$ and the sequences $(a_n)_{n\in\mathbb{N}},(c_n)_{n\in\mathbb{N}}\subseteq(0,\infty)$, $(b_n)_{n\in\mathbb{N}}\subseteq\mathbb{R}$ satisfy $a_n+b_n+c_n=1\;(n\in\mathbb{N})$.\\

Throughout the paper, like in Theorem~\ref{thm:gasper} we use additional appropriate subscripts or superscripts when referring to the Jacobi polynomials $(R_n^{(\alpha,\beta)}(x))_{n\in\mathbb{N}_0}$ or generalized Chebyshev polynomials $(T_n^{(\alpha,\beta)}(x))_{n\in\mathbb{N}_0}$. Moreover, we use an additional superscript ``$+$'' when referring to the sequence $(R_n^{(\alpha,\beta+1)}(x))_{n\in\mathbb{N}_0}$. For instance, there will occur linearization coefficients $g_R(m,n;k)$, $g_R^+(m,n;k)$ and $g_T(m,n;k)$. Observe that a transition from $\beta$ to $\beta+1$, which will play a crucial role in this paper, corresponds to a transition from $(a,b)$ to $(a+1,b-1)$ in the notation of \eqref{eq:abdef}.\\

In the literature, the nonnegativity of all linearization coefficients $g(m,n;k)$ is sometimes called `property (P)'. For the sake of clarity, we shall say `nonnegative linearization of products' throughout the paper. It is well-known that this property enforces the uniqueness of $\mu$, and it is clear that it implies $(a_n)_{n\in\mathbb{N}},(c_n)_{n\in\mathbb{N}}\subseteq(0,1)$ and $(b_n)_{n\in\mathbb{N}}\subseteq[0,1)$. Furthermore, nonnegative linearization of products gives rise to a certain polynomial hypergroup structure, including associated Banach ($L^1$-) algebras and the fruitful possibility of applying Gelfand's theory, which yields a deep and rich harmonic and Fourier analysis \cite{La83}. Hence, nonnegative linearization of products is not only of interest with regard to a better understanding of general or specific orthogonal polynomials, but also has high relevance for functional and abstract harmonic analysis and, in particular, for the theory of Banach algebras. Within such polynomial hypergroups, the classes of Jacobi polynomials and generalized Chebyshev polynomials play a special role concerning product formulas and duality structures \cite{CMS91,CS90,Ga71,Ga72,La80,La83,Ob17}.\footnote{In this paper, the hypergroup context appears only as a kind of additional motivation to study nonnegative linearization of products. In particular, it shows the high relevance for harmonic analysis. The paper can be read without knowledge on hypergroups, however. Roughly speaking, a hypergroup is a generalization of a (locally compact) group which allows the convolution of two Dirac measures to be a probability measure which satisfies certain compatibility and non-degeneracy properties but does not have to be a Dirac measure again (see \cite{BH95,La05} for precise axioms).}

\subsection{Previous results and outline of the paper}

Let us come back to the Jacobi and generalized Chebyshev polynomials. Concerning Theorem~\ref{thm:gasper}, it is not difficult to see that (ii) is necessary for (i). In fact, Gasper has shown that if $b<0$, then $g_R(1,1;1)<0$, whereas if $b\geq0$ and $(\alpha,\beta)\notin V$, then $g_R(2,2;2)<0$ \cite{Ga70b}. The implication ``(ii) $\Rightarrow$ (i)'' is very nontrivial, however. The subcase $(\alpha,\beta)\in\Delta$, where $\Delta\subsetneq V$ is given by
\begin{equation*}
\Delta:=\{(\alpha,\beta)\in(-1,\infty)^2:a,b\geq0\}=\left\{(\alpha,\beta)\in(-1,\infty)^2:\alpha\geq\left|\beta+\frac{1}{2}\right|-\frac{1}{2}\right\},
\end{equation*}
is easier and was already solved in \cite{Ga70a}, and concerning the special case $\alpha\geq\beta\geq-1/2$ Koornwinder gave a less computational proof via addition formulas \cite{Ko78}. Moreover, if $(\alpha,\beta)\in\Delta$, then the nonnegativity of the $g_R(m,n;k)$ can be seen via explicit representations in terms of ${}_9F_8$ hypergeometric series given by Rahman \cite[(1.7) to (1.9)]{Ra81a}.\footnote{The formulas \cite[(1.7) to (1.9)]{Ra81a} contain small mistakes. We will correct them in the appendix (see \eqref{eq:Rahmaneven} to \eqref{eq:Rahmanspecial}).} Alternatively, the case $(\alpha,\beta)\in\Delta$ can also be obtained from one of the aforementioned general criteria of Szwarc \cite{Sz92b}. The simplest subcase is given by $\alpha=\beta\geq-1/2$ (these are the ultraspherical polynomials), for which the nonnegativity of the $g_R(m,n;k)$ follows from Dougall's formula (see \cite[Theorem 6.8.2]{AAR99} or \cite{La05})
\begin{align*}
&R_m^{(\alpha,\alpha)}(x)R_n^{(\alpha,\alpha)}(x)=\\
&=\sum_{j=0}^{\min\{m,n\}}j!\left(\alpha+\frac{1}{2}\right)_j\binom{m}{j}\binom{n}{j}\\
&\quad\times\frac{\left(m+n+\alpha+\frac{1}{2}-2j\right)\left(\alpha+\frac{1}{2}\right)_{m-j}\left(\alpha+\frac{1}{2}\right)_{n-j}(2\alpha+1)_{m+n-j}}{\left(m+n+\alpha+\frac{1}{2}-j\right)\left(\alpha+\frac{1}{2}\right)_{m+n-j}(2\alpha+1)_m(2\alpha+1)_n}R_{m+n-2j}^{(\alpha,\alpha)}(x),
\end{align*}
where $\alpha>-1/2$; for $\alpha=\beta=-1/2$, the linearization coefficients reduce to $1/2$.\\

Besides the original proof given in \cite{Ga70b}, Gasper found a very different one in \cite{Ga83}. The second proof is based on the continuous $q$-Jacobi polynomials and explicit corresponding linearization formulas in terms of ${}_{10}\phi_9$ basic hypergeometric series due to Rahman \cite{Ra81b}. Both \cite{Ga70b} and \cite{Ga83} rely on Descartes' rule of signs (which will be avoided in this paper; cf. the remarks following Lemma~\ref{lma:gRmu}). In the following, we will always refer to the first proof \cite{Ga70b}.\\

For our purposes, it will be more convenient to rewrite $V$ \eqref{eq:Voriginal} as
\begin{equation}\label{eq:Vrewritten}
V=\left\{(\alpha,\beta)\in(-1,\infty)^2:a^2+2b^2+3a\geq3\frac{(a+1)(a+2)}{(a+3)(a+5)}b^2,b\geq0\right\}.
\end{equation}
The small region $V\backslash\Delta$ is bounded on the left by a curve $c$ in the $(\alpha,\beta)$-plane which has the following properties: $c$ starts at the point $(\alpha,\beta)=(-11/8+\sqrt{73}/8,-1)\approx(-0.307,-1)$, approaches the line $\alpha+\beta+1=0$ tangentially and meets this line at the point $(\alpha,\beta)=(-1/2,-1/2)$ (which corresponds to the Chebyshev polynomials of the first kind) \cite{Ga70b}. The angle between the line $\beta=-1$ and $c$ can easily be computed and is given by $\approx87.6^{\circ}$ (in particular, $c$ cannot be written as $\beta=f(\alpha)$ with a single function $f$).\\

Despite the more involved arguments which are required to establish nonnegative linearization of products for $V\backslash\Delta$, from a harmonic analytic point of view there is no reason for restricting to $\Delta$ when studying the associated hypergroups or $L^1$-algebras; we are not aware of a general advantage or benefit a restriction to $\Delta$ would be accompanied with. For instance, in \cite[Theorem 3.1]{Ka15} (or \cite[Theorem 3.1]{Ka16b}) we have shown that the $L^1$-algebra\footnote{in the polynomial hypergroup sense \cite{La83,La05}} associated with $(R_n^{(\alpha,\beta)}(x))_{n\in\mathbb{N}_0}$, $(\alpha,\beta)\in V$, is weakly amenable\footnote{i.e., there exist no nonzero bounded derivations into the dual module $\ell^{\infty}$ (which acts on the $L^1$-algebra via convolution) \cite{La07,La09b}} if and only if $\alpha<0$, and the proof for $\{(\alpha,\beta)\in V:a=0\}\subseteq\Delta$ does not differ from the proof for $V\backslash\Delta$: both cases are traced back to the interior of $\Delta$ via the same argument using inheritance via homomorphisms. This example also shows that important Banach algebraic features like amenability properties may strongly vary even within the same specific class of orthogonal polynomials satisfying nonnegative linearization of products. Hence, also in considering other example classes it is desirable to find various---or even \textit{all}---sequences $(P_n(x))_{n\in\mathbb{N}_0}$ such that nonnegative linearization of products holds (e.g., by characterizing a corresponding parameter range like in Gasper's result Theorem~\ref{thm:gasper} or in our extension to generalized Chebyshev polynomials in Theorem~\ref{thm:mainfull} below).\\

It is an obvious consequence of Theorem~\ref{thm:gasper} and \eqref{eq:genchebeven} that $(T_n^{(\alpha,\beta)}(x))_{n\in\mathbb{N}_0}$ cannot satisfy nonnegative linearization of products if $(\alpha,\beta)\notin V$; moreover, Szwarc has already shown that nonnegative linearization of products is fulfilled for all $(\alpha,\beta)\in\Delta$ \cite{Sz92b,Sz05}. The special case $\alpha\geq\beta+1$ was already shown in \cite{La83}. The simplest subcase is given by $\alpha\geq-1/2\wedge\beta=-1/2$, which can be obtained as above via Dougall's formula; note that $T_n^{(\alpha,-1/2)}(x)=R_n^{(\alpha,\alpha)}(x)$ for all $n\in\mathbb{N}_0$, or in other words: the ultraspherical polynomials are the common subclass of the Jacobi and the generalized Chebyshev polynomials.\\

In Theorem~\ref{thm:mainfull}, we will obtain that $(T_n^{(\alpha,\beta)}(x))_{n\in\mathbb{N}_0}$ satisfies nonnegative linearization of products if and only if $(\alpha,\beta)\in V$. Hence, the generalized Chebyshev polynomials share this property with the Jacobi polynomials. Having in mind the interesting structure of \eqref{eq:genchebodd} (where $\beta+1$ instead of $\beta$ appears on the right-hand side), we will also precisely characterize the pairs $(\alpha,\beta)\in(-1,\infty)^2$ for which all $g_T(m,n;k)$ with at least one odd entry $m,n$ are nonnegative (Theorem~\ref{thm:mainodd}), and we will describe the geometry of the resulting set $V^{\prime}\supsetneq V$.\\

The main results and proofs are given in Section~\ref{sec:jacobi} (Jacobi polynomials) and Section~\ref{sec:gencheb} (generalized Chebyshev polynomials). At several stages, our arguments are based on appropriate decompositions of multivariate polynomials. To find such decompositions (more precisely, appropriate nested sums of suitable factorizations), we also used computer algebra systems (Maple). However, the final proofs can be understood without any computer usage.

\section{Linearization of the product of Jacobi polynomials: sharpening and simplification of Gasper's result}\label{sec:jacobi}

Let $\alpha,\beta>-1$, and let $a,b$ be defined as in Section~\ref{sec:intro}. We first recall that $(R_n^{(\alpha,\beta)}(x))_{n\in\mathbb{N}_0}$ is equivalently given by the orthogonalization measure
\begin{equation*}
\mathrm{d}\mu_R(x)=\frac{\Gamma(\alpha+\beta+2)}{2^{\alpha+\beta+1}\Gamma(\alpha+1)\Gamma(\beta+1)}(1-x)^{\alpha}(1+x)^{\beta}\chi_{(-1,1)}(x)\,\mathrm{d}x
\end{equation*}
and the normalization $R_n^{(\alpha,\beta)}(1)\equiv1$ \cite[Chapter V 2 (B)]{Ch78} \cite[(4.0.2)]{Is09}. $(R_n^{(\alpha,\beta)}(x))_{n\in\mathbb{N}_0}$ satisfies the three-term recurrence relation $R_0^{(\alpha,\beta)}(x)=1$, $R_1^{(\alpha,\beta)}(x)=(x-b_0^R)/a_0^R$,
\begin{equation}\label{eq:recJacobi}
R_1^{(\alpha,\beta)}(x)R_n^{(\alpha,\beta)}(x)=a_n^R R_{n+1}^{(\alpha,\beta)}(x)+b_n^R R_n^{(\alpha,\beta)}(x)+c_n^R R_{n-1}^{(\alpha,\beta)}(x)\;(n\in\mathbb{N})
\end{equation}
with
\begin{equation}\label{eq:reccoeffJacobi}
\begin{split}
a_0^R&=\frac{2\alpha+2}{\alpha+\beta+2}=\frac{a+b+1}{a+1},\\
a_n^R&=\frac{(\alpha+\beta+2)(n+\alpha+\beta+1)(n+\alpha+1)}{(\alpha+1)(2n+\alpha+\beta+1)(2n+\alpha+\beta+2)}=\frac{(a+1)(n+a)(2n+a+b+1)}{(a+b+1)(2n+a)(2n+a+1)}\;(n\in\mathbb{N}),\\
b_0^R&=-\frac{\alpha-\beta}{\alpha+\beta+2}=-\frac{b}{a+1},\\
b_n^R&=\frac{2(\alpha-\beta)n(n+\alpha+\beta+1)}{(\alpha+1)(2n+\alpha+\beta)(2n+\alpha+\beta+2)}=\frac{4b n(n+a)}{(a+b+1)(2n+a-1)(2n+a+1)}\;(n\in\mathbb{N}),\\
c_n^R&\equiv\frac{(\alpha+\beta+2)n(n+\beta)}{(\alpha+1)(2n+\alpha+\beta)(2n+\alpha+\beta+1)}=\frac{(a+1)n(2n+a-b-1)}{(a+b+1)(2n+a-1)(2n+a)}\\
\end{split}
\end{equation}
\cite[(4)]{Ga70a}. It is well-known that
\begin{equation}\label{eq:alphabetachange}
R_n^{(\beta,\alpha)}(x)=(-1)^n\frac{(\alpha+1)_n}{(\beta+1)_n}R_n^{(\alpha,\beta)}(-x)
\end{equation}
\cite[(4.1.4), (4.1.6)]{Is09}.\\

One of our central tools will be a recurrence relation for the $g_R(m,n;k)$ which is taken from \cite{Ga70b} and relies on earlier work of Hylleraas \cite{Hy62}. Let $n\geq m\geq1$; following \cite{Ga70b}, we use a more convenient notation and write
\begin{align*}
n&=m+s,\\
k&=s+j
\end{align*}
with $s\in\mathbb{N}_0$ and $j\in\{0,\ldots,2m\}$. \cite[(2.1)]{Ga70b} states that the linearization coefficients are linked to each other via the following recursion: for $1\leq j\leq2m-1$, one has
\begin{equation}\label{eq:gRrecurrence}
\begin{split}
\theta(m,m+s;j)g_R(m,m+s;s+j+1)&=\iota(m,m+s;j)g_R(m,m+s;s+j)\\
&\quad+\kappa(m,m+s;j)g_R(m,m+s;s+j-1),
\end{split}
\end{equation}
where $\theta(m,m+s;.),\iota(m,m+s;.),\kappa(m,m+s;.):\{1,\ldots,2m-1\}\rightarrow\mathbb{R}$ read
\begin{align}
\label{eq:gRlambda} \theta(m,m+s;j):&=(2m-j+a-1)(2m+2s+j+a+1)\\
\notag &\quad\times\frac{(2s+j+1)(2s+2j+a-b+1)}{(2s+2j+a+1)(2s+2j+a+2)}(j+1),\\
\label{eq:gRmu} \iota(m,m+s;j):&=b\left[(2m-j)(2m+2s+j+2a)\frac{2s+j+1}{2s+2j+a+1}(j+1)\right.\\
\notag &\quad\quad\left.-(2m-j+1)(2m+2s+j+2a-1)\frac{2s+j}{2s+2j+a-1}j\right],\\
\label{eq:gRnu} \kappa(m,m+s;j):&=(2m-j+1)(2m+2s+j+2a-1)\\
\notag &\quad\times\begin{cases} 0, & j-1=s=a=0, \\ \frac{(2s+j+a-1)(2s+2j+a+b-1)}{(2s+2j+a-2)(2s+2j+a-1)}(j+a-1), & \mbox{else}. \end{cases}
\end{align}
Moreover, one has
\begin{align}
\label{eq:gRs} g_R(m,m+s;s)&=\frac{\binom{m+s}{m}\binom{2m+a-1}{m}\binom{m+s+\frac{a-b-1}{2}}{m}}{\binom{2m}{m}\binom{2m+2s+a}{2m}\binom{m+\frac{a+b-1}{2}}{m}},\\
g_R(m,m+s;s+2m)&=\frac{\binom{2m+2s+a-1}{m+s}\binom{2m+a-1}{m}\binom{2m+s+\frac{a+b-1}{2}}{2m+s}}{\binom{4m+2s+a-1}{2m+s}\binom{m+s+\frac{a+b-1}{2}}{m+s}\binom{m+\frac{a+b-1}{2}}{m}}
\end{align}
and
\begin{align}
\label{eq:gRs1} &g_R(m,m+s;s+1)=\\
\notag &=\frac{4b m(m+s+a)(2s+a+2)}{(2m+2s+a+1)(2m+a-1)(2s+a-b+1)}g_R(m,m+s;s),\\
\label{eq:gR2ms1} &g_R(m,m+s;s+2m-1)=\\
\notag &=\frac{4b m(m+s)(4m+2s+a-2)}{(4m+2s+a+b-1)(2m+2s+a-1)(2m+a-1)}g_R(m,m+s;s+2m)
\end{align}
\cite[(2.2) to (2.4), (2.9)]{Ga70b}. The following auxiliary result deals with the canonical continuation of the coefficient function $\iota(m,m+s;.)$ to $[1,2m-1]$ and can be seen from \cite[Section 2]{Ga70a} and \cite[Section 2]{Ga70b}:

\begin{lemma}\label{lma:gRmu}
Let $(\alpha,\beta)\in V$ with $\alpha\neq\beta$, and let for $m\in\mathbb{N}$ and $s\in\mathbb{N}_0$ the function $\iota(m,m+s;.):[1,2m-1]\rightarrow\mathbb{R}$ be defined by \eqref{eq:gRmu}. Then $\iota(m,m+s;.)$ has at most one zero. Moreover, if $m\geq2$ or $a\geq0$, then $\iota(m,m+s;1)\geq0$.
\end{lemma}

The proofs given in \cite{Ga70a,Ga70b} rely on Descartes' rule of signs. In the appendix, we present an alternative proof of Lemma~\ref{lma:gRmu}. It will completely avoid Descartes' rule of signs and use the classical mean value theorem instead. The proof of Theorem~\ref{thm:gasperpositivevariant} below, which is the main result of this section, will essentially rely on Lemma~\ref{lma:gRmu}. Concerning the functions $\theta(m,m+s;.)$ and $\kappa(m,m+s;.)$, we will only need that 
\begin{equation}\label{eq:thetapos}
\theta(m,m+s;.)|_{\{1,\ldots,2m-2\}}>0\;(m\geq2)
\end{equation}
and
\begin{equation}\label{eq:kappapos}
\kappa(m,m+s;.)|_{\{2,\ldots,2m-1\}}>0\;(m\geq2),
\end{equation}
which is an obvious consequence of \eqref{eq:gRlambda} and \eqref{eq:gRnu} and was also used in \cite{Ga70b}.\\

We now give two characterizations. The first deals with positivity of the linearization coefficients $g_R(m,n;k)$ and can be regarded as a sharpening of Theorem~\ref{thm:gasper} because the nontrivial direction ``(ii) $\Rightarrow$ (i) (Theorem~\ref{thm:gasper})'' is implied by ``(ii) $\Rightarrow$ (i) (Theorem~\ref{thm:gasperpositivevariant})'' via continuity. Our proof is a---considerably less computational---modification of Gasper's approach \cite{Ga70b}. The second characterization follows from the first and deals with a certain oscillatory behavior of the $g_R(m,n;k)$. It will play an important role for the proof of Theorem~\ref{thm:mainodd} on generalized Chebyshev polynomials below.

\begin{theorem}\label{thm:gasperpositivevariant}
Let $\alpha,\beta>-1$. The following are equivalent:
\begin{enumerate}
\item[\rm(i)] All $g_R(m,n;k)$ are positive.
\item[\rm(ii)] $(\alpha,\beta)$ is located in the interior of $V$ (denoted by $V^{\circ}$ in the following), i.e.,
\begin{equation}\label{eq:Vinteriorfirst}
a^2+2b^2+3a>3\frac{(a+1)(a+2)}{(a+3)(a+5)}b^2
\end{equation}
and
\begin{equation}\label{eq:Vinteriorsecond}
b>0.
\end{equation}
\end{enumerate}
\end{theorem}

\begin{corollary}\label{cor:oscillatory}
Let $\alpha,\beta>-1$, and let $\widetilde{g_R}(m,n;k)$ denote the linearization coefficients belonging to the sequence $(R_n^{(\beta,\alpha)}(x))_{n\in\mathbb{N}_0}$. Then the following hold:
\begin{enumerate}
\item[\rm(i)] All numbers $(-1)^{m+n+k}\widetilde{g_R}(m,n;k)$ are nonnegative if and only if $(\alpha,\beta)\in V$.
\item[\rm(ii)] All numbers $(-1)^{m+n+k}\widetilde{g_R}(m,n;k)$ are positive if and only if $(\alpha,\beta)\in V^{\circ}$.
\end{enumerate}
\end{corollary}

\begin{proof}[Proof (Theorem~\ref{thm:gasperpositivevariant})]
We preliminarily note that $V^{\circ}$ can indeed be characterized as the set of all $(\alpha,\beta)\in(-1,\infty)^2$ satisfying the strict inequalities \eqref{eq:Vinteriorfirst} and \eqref{eq:Vinteriorsecond} because the gradient of the function $\varphi:(-1,\infty)^2\rightarrow\mathbb{R}$,
\begin{equation*}
\varphi(\alpha,\beta):=(a^2+2b^2+3a)(a+3)(a+5)-3(a+1)(a+2)b^2
\end{equation*}
does not vanish on $V$ (and obviously the same is the case for the function $(-1,\infty)^2\rightarrow\mathbb{R}$, $(\alpha,\beta)\mapsto b$); use that every $(\alpha,\beta)\in V$ satisfies
\begin{align*}
\frac{\partial\varphi}{\partial\alpha}(\alpha,\beta)&\geq\frac{\partial\varphi}{\partial\alpha}(\alpha,\beta)-(a^2+2b^2+3a)(4\beta+20)=\\
&=(24\beta^2+100\beta+116)\alpha-8\beta^3-40\beta^2-20\beta+80\geq\\
&\geq(24\beta^2+100\beta+116)\beta-8\beta^3-40\beta^2-20\beta+80=\\
&=(4\beta+8)(4\beta^2+7\beta+10)>\\
&>0.
\end{align*}

We establish the easy direction ``(i) $\Rightarrow$ (ii)'' in a similar way as in \cite{Ga70b}: if $b\leq0$, then \eqref{eq:recJacobi} and \eqref{eq:reccoeffJacobi} (or also \eqref{eq:gRs} and \eqref{eq:gRs1}) show that
\begin{equation*}
g_R(1,1;1)=b_1^R=\frac{4b}{(a+3)(a+b+1)}\leq0;
\end{equation*}
if $b>0$ but $(\alpha,\beta)$ is not located in the interior of $V$, then the equations \eqref{eq:gRrecurrence} to \eqref{eq:gRs} and \eqref{eq:gRs1} yield
\begin{equation*}
g_R(2,2;2)=\frac{4[(a^2+2b^2+3a)(a+3)(a+5)-3(a+1)(a+2)b^2]}{(a+3)(a+5)(a+6)(a+b+1)(a+b+3)}\leq0.
\end{equation*}
We now come to the interesting direction ``(ii) $\Rightarrow$ (i)'', so let $(\alpha,\beta)\in V^{\circ}$, let $m\in\mathbb{N}$ and let $s\in\mathbb{N}_0$. We have to show that $g_R(m,m+s;s+j)>0$ for all $j\in\{0,\ldots,2m\}$. Starting similarly to \cite{Ga70b}, we use ``two-sided induction'' and proceed as follows: \eqref{eq:gRs} to \eqref{eq:gR2ms1} yield $g_R(m,m+s;s)>0$, $g_R(m,m+s;s+1)>0$, $g_R(m,m+s;s+2m)>0$ and $g_R(m,m+s;s+2m-1)>0$.\footnote{Of course, the positivity of $g_R(m,m+s;s)$ and $g_R(m,m+s;s+2m)$ is also clear from general results, cf. Section~\ref{sec:intro}.} If $m=1$, we are already done (this case is also clear from \eqref{eq:recJacobi} and \eqref{eq:reccoeffJacobi}). Hence, assume that $m\geq2$ from now on; it is then left to show that $g_R(m,m+s;s+j)>0$ for all $j\in\{2,\ldots,2m-2\}$.\\

\eqref{eq:thetapos} and \eqref{eq:kappapos} yield $\theta(m,m+s;1)>0$ and $\kappa(m,m+s;2m-1)>0$, and via the equations \eqref{eq:gRrecurrence} to \eqref{eq:gRnu} and \eqref{eq:gRs1}, \eqref{eq:gR2ms1} we compute
\begin{equation}\label{eq:simplifygasperI}
\begin{split}
&\frac{(2m+a-1)(2s+a-b+1)(2m+2s+a+1)(2s+a+3)}{4m(m+s+a)(2s+a+1)g_R(m,m+s;s)}\\
&\quad\times\theta(m,m+s;1)g_R(m,m+s;s+2)=\\
&=\frac{(2m+a-1)(2s+a-b+1)(2m+2s+a+1)(2s+a+3)}{4m(m+s+a)(2s+a+1)g_R(m,m+s;s)}\\
&\quad\times[\iota(m,m+s;1)g_R(m,m+s;s+1)+\kappa(m,m+s;1)g_R(m,m+s;s)]=\\
&=(b^2+a)(2m-4)(2m+2s+2a+4)2s\\
&\quad+(a^2+2b^2+3a)[(2m-4)(2m+2s+2a+4)+2s(2s+2a+8)+(a+3)(a+5)]\\
&\quad-3(a+1)(a+2)b^2
\end{split}
\end{equation}
and
\begin{equation}\label{eq:simplifygasperII}
\begin{split}
&\frac{(2m+a-1)(2m+2s+a-1)(4m+2s+a-3)(4m+2s+a+b-1)}{4m(m+s)(4m+2s+a-1)g_R(m,m+s;s+2m)}\\
&\quad\times\kappa(m,m+s;2m-1)g_R(m,m+s;s+2m-2)=\\
&=\frac{(2m+a-1)(2m+2s+a-1)(4m+2s+a-3)(4m+2s+a+b-1)}{4m(m+s)(4m+2s+a-1)g_R(m,m+s;s+2m)}\\
&\quad\times[\theta(m,m+s;2m-1)g_R(m,m+s;s+2m)\\
&\quad\quad-\iota(m,m+s;2m-1)g_R(m,m+s;s+2m-1)]=\\
&=(b^2+a)(2m-4)(2m+2s-4)(4m+2s+2a)\\
&\quad+(a^2+2b^2+3a)[(2m-4)(6m+6s+4a+4)+2s(2s+2a+8)+(a+3)(a+5)]\\
&\quad-3(a+1)(a+2)b^2.
\end{split}
\end{equation}
Observe that $b^2+a>0$: if $a\leq0$, this is a consequence of the decomposition
\begin{equation*}
2b^2+2a=a^2+2b^2+3a-a(a+1).
\end{equation*}
Therefore, the right hand sides of \eqref{eq:simplifygasperI} and \eqref{eq:simplifygasperII} are greater than or equal to $(a^2+2b^2+3a)(a+3)(a+5)-3(a+1)(a+2)b^2$, so these equations imply that $g_R(m,m+s;s+2)>0$ and $g_R(m,m+s;s+2m-2)>0$. We note at this stage that \eqref{eq:simplifygasperI} and \eqref{eq:simplifygasperII} allow us to obtain the positivity of $g_R(m,m+s;s+2)$ and $g_R(m,m+s;s+2m-2)$ in a much faster way than Gasper estimated in \cite{Ga70b} (establishing the nonnegativity of $g_R(m,m+s;s+2)$ and $g_R(m,m+s;s+2m-2)$ under the assumption $(\alpha,\beta)\in V$). Besides the avoidance of Descartes' rule of signs in an alternative proof of Lemma~\ref{lma:gRmu} given in the appendix (the lemma will be used in the induction step below), this is our essential simplification of Gasper's approach. A more detailed comparison will be given subsequent to the proof.\\

As we are done if $m=2$, we assume that $m\geq3$ from now on. The remaining proof works similarly as in \cite{Ga70b} again. We first apply Lemma~\ref{lma:gRmu} and obtain the existence of an $N\in\{1,\ldots,2m-1\}$ such that $\iota(m,m+s;j)\geq0$ for $1\leq j\leq N$ and $\iota(m,m+s;j)<0$ for those $1\leq j\leq2m-1$ which satisfy $j\geq N+1$. We then make use of \eqref{eq:thetapos} and \eqref{eq:kappapos} and distinguish two cases:\\

\textit{Case 1:} $N\geq3$. Then \eqref{eq:gRrecurrence} and induction yield
\begin{align*}
&g_R(m,m+s;s+j+1)=\\
&=\frac{\iota(m,m+s;j)}{\theta(m,m+s;j)}g_R(m,m+s;s+j)+\frac{\kappa(m,m+s;j)}{\theta(m,m+s;j)}g_R(m,m+s;s+j-1)>\\
&>0\;(2\leq j\leq N-1).
\end{align*}
This shows the positivity of $g_R(m,m+s;s+3),\ldots,g_R(m,m+s;s+N)$.\\

\textit{Case 2:} $N\leq2m-3$. In this case, \eqref{eq:gRrecurrence} and induction yield
\begin{align*}
&g_R(m,m+s;s+j-1)=\\
&=-\frac{\iota(m,m+s;j)}{\kappa(m,m+s;j)}g_R(m,m+s;s+j)+\frac{\theta(m,m+s;j)}{\kappa(m,m+s;j)}g_R(m,m+s;s+j+1)>\\
&>0\;(N+1\leq j\leq2m-2),
\end{align*}
which establishes the positivity of $g_R(m,m+s;s+N),\ldots,g_R(m,m+s;s+2m-3)$.\\

If $N\leq2$, then $N<2m-3$ and the positivity of
\begin{equation*}
g_R(m,m+s;s+3),\ldots,g_R(m,m+s;s+2m-3)
\end{equation*}
is a consequence of Case 2. If $N\geq2m-2$, then $N>3$ and the positivity of $g_R(m,m+s;s+3),\ldots,g_R(m,m+s;s+2m-3)$ is a consequence of Case 1. Finally, if $3\leq N\leq2m-3$, then the combination of both cases yields the positivity of $g_R(m,m+s;s+3),\ldots,g_R(m,m+s;s+2m-3)$.
\end{proof}

Our argument via the central equations \eqref{eq:simplifygasperI} and \eqref{eq:simplifygasperII} in the initial step above shows a typical aspect of the strategy (this aspect will also be important in Section~\ref{sec:gencheb}): as soon as one knows decompositions like in \eqref{eq:simplifygasperI}, \eqref{eq:simplifygasperII} which allow one to see the signs of the relevant parts, these decompositions can easily (yet more or less tediously) be verified by comparing the expansions, or by comparing common zeros and leading coefficients and so on. Hence, the actual task is \textit{finding} such decompositions. In this context, we want to make some additional notes concerning a comparison between Gasper's strategy \cite{Ga70b} and ours: instead of our decomposition \eqref{eq:simplifygasperI}, Gasper decomposed the left-hand side by finding
\begin{equation}\label{eq:simplifygasperIgaspervariant}
\begin{split}
&\frac{(2m+a-1)(2s+a-b+1)(2m+2s+a+1)(2s+a+3)}{4m(m+s+a)(2s+a+1)g_R(m,m+s;s)}\\
&\quad\times\theta(m,m+s;1)g_R(m,m+s;s+2)=\\
&=\frac{a A(a)+b^2B(a)}{(2s+a+1)^2}
\end{split}
\end{equation}
with
\begin{equation*}
A(a):=(2m+2s+a+1)(2m+a-1)(2s+a+3)(2s+a+1)^2\;(>0)
\end{equation*}
and
\begin{align*}
B(a):&=-a(2m+2s+a+1)(2m+a-1)(2s+a+3)\\
&\quad+4(2s+a+2)[(s+1)(2m-1)(2m+2s+2a+1)(2s+a+1)\\
&\quad\quad\quad\quad\quad\quad\quad\quad-m(m+s+a)(2s+1)(2s+a+3)].
\end{align*}
Using the original formulas contained in \eqref{eq:Voriginal}, Gasper then estimated
\begin{align*}
&\frac{(2m+a-1)(2s+a-b+1)(2m+2s+a+1)(2s+a+3)}{4m(m+s+a)(2s+a+1)g_R(m,m+s;s)}\\
&\quad\times\theta(m,m+s;1)g_R(m,m+s;s+2)\geq\\
&\geq\frac{b^2D(a)}{(a+5)(a+3)^2(2s+a+1)^2}
\end{align*}
for all $(\alpha,\beta)\in V$, where
\begin{equation*}
D(a):=(a^2-7a-24)A(a)+(a+5)(a+3)^2B(a)=4\sum_{k=0}^6d_k a^k
\end{equation*}
for some $d_0,\ldots,d_6\in\mathbb{R}$ which depend on $m,s$ but do not depend on $a,b$. Finally, Gasper explicitly computed the cumbersome coefficients $d_0,\ldots,d_6$, used that $a>-1/3$ for all $(\alpha,\beta)\in V$, restricted to the special case $a<0$ (cf. Remark~\ref{rem:simplagreaterzero} below) and estimated $D(a)\geq0\;(a\in(-1/3,0))$ by showing that $D(-1/3)\geq0$, $D^{\prime}(-1/3)\geq0$ and $D^{\prime\prime}(a)\geq0\;(a\in[-1/3,0])$ to conclude that $g_R(m,m+s;s+2)\geq0$.\footnote{The prime means the derivative w.r.t. $a$.} In contrast to this rather painful decomposition \eqref{eq:simplifygasperIgaspervariant} and the original formula for $V$ \eqref{eq:Voriginal}, our decomposition \eqref{eq:simplifygasperI} and the rewritten form of $V$ \eqref{eq:Vrewritten} allow us to directly see that $g_R(m,m+s;s+2)\geq0$ for all $(\alpha,\beta)\in V$ (and ,,$>0$'' for $(\alpha,\beta)\in V^{\circ}$). Similar differences occur concerning the linearization coefficient $g_R(m,m+s;s+2m-2)$.

\begin{remark}\label{rem:simplagreaterzero}
As was similarly observed in \cite{Ga70a,Ga70b}, the proof of the direction ``(ii) $\Rightarrow$ (i)'' considerably simplifies in the special case $a>0$, i.e., for $(\alpha,\beta)\in\Delta^{\circ}\subsetneq V^{\circ}$. On the one hand, for $a>0$ the functions $\theta(m,m+s;.)$ and $\kappa(m,m+s;.)$ are positive on their full domains (cf. \eqref{eq:gRlambda} and \eqref{eq:gRnu}). Hence, one can avoid the computations (and appropriate decompositions) of $g_R(m,m+s;s+2)$ and $g_R(m,m+s;s+2m-2)$. On the other hand, the proof of the important ingredient Lemma~\ref{lma:gRmu} is simpler for $a>0$, cf. also our alternative proof of this lemma given in the appendix. If $(\alpha,\beta)$ is located in the interior of $\Delta$, then the positivity of all $g_R(m,n;k)$ can also be seen via Rahman's formulas \eqref{eq:Rahmaneven} and \eqref{eq:Rahmanodd}. The simplest subcase is given by $\alpha=\beta+1$, for which the positivity of the $g_R(m,n;k)$ can be seen via a very simple explicit formula \cite{Hy62}.
\end{remark}

\begin{proof}[Proof (Corollary~\ref{cor:oscillatory})]
As a consequence of \eqref{eq:alphabetachange}, the linearization coefficients are connected to each other via
\begin{equation*}
(-1)^{m+n+k}\widetilde{g_R}(m,n;k)=\frac{(\alpha+1)_m(\alpha+1)_n}{(\alpha+1)_k}\frac{(\beta+1)_k}{(\beta+1)_m(\beta+1)_n}g_R(m,n;k).
\end{equation*}
Hence, the assertions are immediate consequences of Theorem~\ref{thm:gasper} and Theorem~\ref{thm:gasperpositivevariant}, cf. also the remarks at the end of \cite[Section 1]{Ga70b}.
\end{proof}

\section{Linearization of the product of generalized Chebyshev polynomials: solution to a problem of Szwarc}\label{sec:gencheb}

Let $\alpha,\beta>-1$ again, and let $a,b$ be defined as in Section~\ref{sec:intro}. In the following, we use the notation and auxiliary functions of Section~\ref{sec:jacobi}. The sequence $(T_n^{(\alpha,\beta)}(x))_{n\in\mathbb{N}_0}$ of generalized Chebyshev polynomials\footnote{Some authors prefer to call these polynomials `generalized ultraspherical polynomials' or `generalized Gegenbauer polynomials', and some authors use the expression `generalized Chebyshev polynomials' for different things.} is also (uniquely) determined by the orthogonalization measure
\begin{equation*}
\mathrm{d}\mu_T(x)=\frac{\Gamma(\alpha+\beta+2)}{\Gamma(\alpha+1)\Gamma(\beta+1)}(1-x^2)^{\alpha}|x|^{2\beta+1}\chi_{(-1,1)}(x)\,\mathrm{d}x
\end{equation*}
and the normalization $T_n^{(\alpha,\beta)}(1)\equiv1$ \cite[Chapter V 2 (G)]{Ch78} \cite[(4.0.2)]{Is09}. Moreover, $(T_n^{(\alpha,\beta)}(x))_{n\in\mathbb{N}_0}$ satisfies the recurrence relation $T_0^{(\alpha,\beta)}(x)=1$, $T_1^{(\alpha,\beta)}(x)=x$,
\begin{equation}\label{eq:recgencheb}
x T_n^{(\alpha,\beta)}(x)=a_n^T T_{n+1}^{(\alpha,\beta)}(x)+c_n^T T_{n-1}^{(\alpha,\beta)}(x)\;(n\in\mathbb{N})
\end{equation}
with $(a_n^T)_{n\in\mathbb{N}},(c_n^T)_{n\in\mathbb{N}}\subseteq(0,1)$ given by
\begin{equation}\label{eq:reccoeffgencheb}
\begin{split}
a_{2n-1}^T&\equiv\frac{n+\alpha}{2n+\alpha+\beta}=\frac{2n+a+b-1}{4n+2a-2},\;a_{2n}^T\equiv\frac{n+\alpha+\beta+1}{2n+\alpha+\beta+1}=\frac{n+a}{2n+a},\\
c_{2n-1}^T&\equiv\frac{n+\beta}{2n+\alpha+\beta}=\frac{2n+a-b-1}{4n+2a-2},\;c_{2n}^T\equiv\frac{n}{2n+\alpha+\beta+1}=\frac{n}{2n+a}
\end{split}
\end{equation}
\cite[3 (f)]{La83}. Using \eqref{eq:genchebeven}, \eqref{eq:genchebodd}, \eqref{eq:recgencheb} and \eqref{eq:reccoeffgencheb}, one can relate the $g_T(m,n;k)$ to the $g_R(m,n;k)$ and $g_R^+(m,n;k)$. For instance, this was done in \cite{La83}: one has
\begin{equation}\label{eq:gTeven}
g_T(2m,2n;2k)=g_R(m,n;k)
\end{equation}
and
\begin{equation}\label{eq:gTodd}
g_T(2m+1,2n+1;2k)=\begin{cases} c_{2|m-n|+1}^T g_R^+(m,n;|m-n|), & k=|m-n|, \\ a_{2m+2n+1}^T g_R^+(m,n;m+n), & k=m+n+1, \\ a_{2k-1}^T g_R^+(m,n;k-1)+c_{2k+1}^T g_R^+(m,n;k), & \mbox{else}; \end{cases}
\end{equation}
moreover, $g_T(m,n;k)=0$ if $m+n-k$ is odd (trivial consequence of symmetry), and $g_T(2m+1,2n;2k+1)$ and $g_T(2m,2n+1;2k+1)$ relate to \eqref{eq:gTodd} via
\begin{equation}\label{eq:gTmixed}
g_T(2m+1,2n;2k+1)=\frac{h_T(2k+1)}{h_T(2n)}g_T(2m+1,2k+1;2n),
\end{equation}
which is a consequence of \eqref{eq:lincoeffh}, and
\begin{equation}\label{eq:gTadd}
g_T(2m,2n+1;2k+1)=g_T(2n+1,2m;2k+1).
\end{equation}
For the sake of completeness, we recall the short proof of \eqref{eq:gTodd}: just expand
\begin{align*}
&T_{2m+1}^{(\alpha,\beta)}(x)T_{2n+1}^{(\alpha,\beta)}(x)=\\
&=x^2R_m^{(\alpha,\beta+1)}(2x^2-1)R_n^{(\alpha,\beta+1)}(2x^2-1)=\\
&=x^2\sum_{k=|m-n|}^{m+n}g_R^+(m,n;k)R_k^{(\alpha,\beta+1)}(2x^2-1)=\\
&=x\sum_{k=|m-n|}^{m+n}g_R^+(m,n;k)T_{2k+1}^{(\alpha,\beta)}(x)=\\
&=\sum_{k=|m-n|}^{m+n}g_R^+(m,n;k)[a_{2k+1}^T T_{2k+2}^{(\alpha,\beta)}(x)+c_{2k+1}^T T_{2k}^{(\alpha,\beta)}(x)]=\\
&=\sum_{k=|m-n|+1}^{m+n+1}a_{2k-1}^T g_R^+(m,n;k-1)T_{2k}^{(\alpha,\beta)}(x)+\sum_{k=|m-n|}^{m+n}c_{2k+1}^T g_R^+(m,n;k)T_{2k}^{(\alpha,\beta)}(x).
\end{align*}
The proof of \eqref{eq:gTeven} is even shorter.\\

We deal with the following problems:
\begin{enumerate}
\item[\rm(A)] Szwarc's problem, cf. Section~\ref{sec:intro}: find all pairs $(\alpha,\beta)\in(-1,\infty)^2$ such that $(T_n^{(\alpha,\beta)}(x))_{n\in\mathbb{N}_0}$ satisfies nonnegative linearization of products, i.e., such that all $g_T(m,n;k)$ are nonnegative.
\item[\rm(B)] Find all pairs $(\alpha,\beta)\in(-1,\infty)^2$ such that all $g_T(m,n;k)$ with at least one odd entry $m,n$ are nonnegative.
\end{enumerate}
The pairs $(\alpha,\beta)\in(-1,\infty)^2$ such that all $g_T(m,n;k)$ with two even entries $m,n$ are nonnegative are exactly the $(\alpha,\beta)\in V$, which is an obvious consequence of \eqref{eq:gTeven} and Theorem~\ref{thm:gasper}. Hence, it will be interesting to compare the resulting set of (B) to $V$.\\

The solutions to (A) and (B) will be given in Theorem~\ref{thm:mainfull} and Theorem~\ref{thm:mainodd}, respectively. We want to motivate these results by establishing two necessary conditions for the pairs $(\alpha,\beta)$ which are as in (B): given any $\alpha,\beta>-1$ and arbitrary $m,n\in\mathbb{N}$ with $n\geq m$, we use the notation of the previous sections and compute
\begin{equation}\label{eq:necfirst}
\begin{split}
&(2m+a)(2m+2s+a+2)\frac{2s+a+b+1}{2s+a+2}\left(\frac{c_{2s+3}^T}{a_{2s+1}^T}\frac{g_R^+(m,m+s;s+1)}{g_R^+(m,m+s;s)}+1\right)=\\
&=4bm^2+4b(s+a+1)m+a(2s+a+b+1)
\end{split}
\end{equation}
via \eqref{eq:gRs1} and \eqref{eq:reccoeffgencheb}. Making also use of \eqref{eq:gRrecurrence}, which yields
\begin{equation*}
\frac{g_R^+(m,m+s;s+2)}{\underbrace{g_R^+(m,m+s;s+1)}_{\neq0}}=\frac{\iota^+(m,m+s;1)}{\underbrace{\theta^+(m,m+s;1)}_{>0}}+\frac{\kappa^+(m,m+s;1)}{\theta^+(m,m+s;1)}\frac{g_R^+(m,m+s;s)}{g_R^+(m,m+s;s+1)}
\end{equation*}
for $b\neq1$, and combining this with \eqref{eq:gRlambda} to \eqref{eq:gRnu}, \eqref{eq:gRs1} and \eqref{eq:reccoeffgencheb}, we furthermore obtain
\begin{equation}\label{eq:necsecond}
\begin{split}
&4(b-1)(2m+a-1)(2m+2s+a+3)\frac{(s+1)(2s+a+b+3)}{2s+a+4}\\
&\quad\times\left(\frac{c_{2s+5}^T}{a_{2s+3}^T}\frac{g_R^+(m,m+s;s+2)}{g_R^+(m,m+s;s+1)}+1\right)=\\
&=(4m-4)(m+s+a+2)\left[(a^2+2b^2+3a)(s+1)-a(a+1)s\right]\\
&\quad+(a+1)(2s+a+b+3)\left[(a+2b)(2s+2-b)+a^2+2b^2+3a\right]\;(b\neq1).
\end{split}
\end{equation}
If $b<0$, then the right-hand side of \eqref{eq:necfirst} becomes negative for (all) sufficiently large $m\in\mathbb{N}$, whereas
\begin{equation*}
(2m+a)(2m+2s+a+2)\frac{2s+a+b+1}{2s+a+2}
\end{equation*}
is always positive. Hence, if $b<0$, then
\begin{equation*}
\frac{c_{2s+3}^T}{a_{2s+1}^T}\frac{g_R^+(m,m+s;s+1)}{g_R^+(m,m+s;s)}+1
\end{equation*}
is negative for sufficiently large $m\in\mathbb{N}$. Since $g_R^+(m,m+s;s)$ is always positive, the latter yields the negativity of
\begin{equation*}
g_T(2m+1,2m+2s+1;2s+2)=a_{2s+1}^T g_R^+(m,m+s;s)+c_{2s+3}^T g_R^+(m,m+s;s+1)
\end{equation*}
\eqref{eq:gTodd} for sufficiently large $m\in\mathbb{N}$.\\

Now assume that $a^2+2b^2+3a<0$. On the one hand, one necessarily has $b<1$ then (because $a^2+2+3a=(a+1)(a+2)>0$), so
\begin{equation*}
4(b-1)(2m+a-1)(2m+2s+a+3)\frac{(s+1)(2s+a+b+3)}{2s+a+4}
\end{equation*}
is always negative. On the other hand, if $s=0$, then the right-hand side of \eqref{eq:necsecond} becomes negative for (all) sufficiently large $m\in\mathbb{N}$. Hence,
\begin{equation*}
\frac{c_5^T}{a_3^T}\frac{g_R^+(m,m;2)}{g_R^+(m,m;1)}+1
\end{equation*}
is positive for sufficiently large $m\in\mathbb{N}$. Since, due to \eqref{eq:gRs1}, $g_R^+(m,m;1)$ is negative, we obtain the negativity of
\begin{equation*}
g_T(2m+1,2m+1;4)=a_3^T g_R^+(m,m;1)+c_5^T g_R^+(m,m;2)
\end{equation*}
\eqref{eq:gTodd} for sufficiently large $m\in\mathbb{N}$.\\

Putting all together, we see that every pair $(\alpha,\beta)$ which fits in (B) has to satisfy both $b\geq0$ and $a^2+2b^2+3a\geq0$. Our following result deals with the converse and shows that these two conditions already \textit{characterize} (B).

\begin{theorem}\label{thm:mainodd}
Let $\alpha,\beta>-1$. The following are equivalent:
\begin{enumerate}
\item[\rm(i)] For all $m,n\in\mathbb{N}_0$ such that at least one of these numbers is odd, all linearization coefficients $g_T(m,n;k)$ are nonnegative.
\item[\rm(ii)] $(\alpha,\beta)\in V^{\prime}$, where
\begin{equation*}
V^{\prime}:=\left\{(\alpha,\beta)\in(-1,\infty)^2:a^2+2b^2+3a\geq0,b\geq0\right\}\supsetneq V.
\end{equation*}
\end{enumerate}
If $(\alpha,\beta)\in V^{\prime}\backslash\Delta$ and $m,n\in\mathbb{N}_0$ are such that at least one of these numbers is odd, and if $k\in\{|m-n|,\ldots,m+n\}$ is such that $m+n-k$ is even, then $g_T(m,n;k)$ is positive.
\end{theorem}

The inclusion $V^{\prime}\supsetneq V$ is clear from the rewritten form of $V$ \eqref{eq:Vrewritten} (one can easily find points which show that the inclusion is proper). $V^{\prime}$ is a subset of $[-1/2,\infty)\times(-1,\infty)$, which can be seen as follows: let $(\alpha,\beta)\in V^{\prime}$. If $a\geq0$, then $(\alpha,\beta)\in\Delta\subseteq[-1/2,\infty)\times(-1,\infty)$, and if $a<0$, then
\begin{equation*}
2\underbrace{(b-a)}_{>0}(b+a)=a^2+2b^2+3a-3a(a+1)>0
\end{equation*}
and consequently
\begin{equation*}
0<b+a=2\alpha+1.
\end{equation*}
This establishes $V^{\prime}\subseteq[-1/2,\infty)\times(-1,\infty)$; it is clear that the inclusion is proper. Concerning the geometry of $V^{\prime}$, we note that one obtains the set $\left\{(\alpha,\beta)\in\mathbb{R}^2:a^2+2b^2+3a=0\right\}$ by rotating the ellipse
\begin{equation*}
\left(\frac{x}{\frac{3}{4}\sqrt{2}}\right)^2+\left(\frac{y}{\frac{3}{4}}\right)^2=1
\end{equation*}
by $\pi/4$ and shifting the image by $(-5/4,-5/4)^T$. The small region $V^{\prime}\backslash V$ is bounded on the left by a curve $c^{\prime}$ in the $(\alpha,\beta)$-plane which starts at the point $(\alpha,\beta)=(-1/3,-1)$, approaches the line $\alpha+\beta+1=0$ tangentially and meets this line at the point $(\alpha,\beta)=(-1/2,-1/2)$ (cf. also the related set $W$ considered in \cite{Ga70b}). The angle between the line $\beta=-1$ and $c^{\prime}$ is $\approx86.2^{\circ}$ (in particular, $c^{\prime}$ cannot be written as $\beta=f(\alpha)$ with a single function $f$).\\

As a consequence of Theorem~\ref{thm:mainodd}, we will obtain our second main result of this section and the answer to (A):

\begin{theorem}\label{thm:mainfull}
Let $\alpha,\beta>-1$. The following are equivalent:
\begin{enumerate}
\item[\rm(i)] $(T_n^{(\alpha,\beta)}(x))_{n\in\mathbb{N}_0}$ satisfies nonnegative linearization of products, i.e., all $g_T(m,n;k)$ are nonnegative.
\item[\rm(ii)] $(\alpha,\beta)\in V$.
\end{enumerate}
\end{theorem}

With respect to Theorem~\ref{thm:mainfull}, we note that if $b\geq0$ and $(\alpha,\beta)\notin V$, then $g_T(4,4;4)<0$, which is a consequence of $g_R(2,2;2)<0$ (cf. Section~\ref{sec:intro}) and \eqref{eq:gTeven}. In the same way, we see that if $b<0$, then $g_T(2,2;2)<0$. Theorem~\ref{thm:mainfull} can be regarded as a another sharpening of Gasper's result Theorem~\ref{thm:gasper} because the nontrivial direction ``(ii) $\Rightarrow$ (i) (Theorem~\ref{thm:gasper})'' is trivially implied by ``(ii) $\Rightarrow$ (i) (Theorem~\ref{thm:mainfull})'' and \eqref{eq:gTeven}. As a consequence of Theorem~\ref{thm:mainfull}, we obtain that $(T_n^{(\alpha,\beta)}(x))_{n\in\mathbb{N}_0}$ induces a polynomial hypergroup and an associated $L^1$-algebra \cite{La83} whenever $(\alpha,\beta)\in V$.\\

Comparing Theorem~\ref{thm:gasper}, Theorem~\ref{thm:gasperpositivevariant} and Theorem~\ref{thm:mainfull}, one may ask whether all $g_T(m,n;k)$ are positive if $(\alpha,\beta)$ is located in the interior of $V$. However, this is not true: just recall that for every choice of $(\alpha,\beta)\in(-1,\infty)^2$ one has $g_T(m,n;k)=0$ if $m+n-k$ is odd.\\

We now come to the proofs.\\

Since Theorem~\ref{thm:mainodd} implies that the set of all pairs $(\alpha,\beta)\in(-1,\infty)^2$ such that $(T_n^{(\alpha,\beta)}(x))_{n\in\mathbb{N}_0}$ satisfies nonnegative linearization of products is given by $V\cap V^{\prime}=V$, Theorem~\ref{thm:mainfull} follows from Theorem~\ref{thm:mainodd}.\\

The implication ``(i) $\Rightarrow$ (ii)'' of Theorem~\ref{thm:mainodd} was already established above. In view of Szwarc's earlier result, which already shows that $(T_n^{(\alpha,\beta)}(x))_{n\in\mathbb{N}_0}$ satisfies nonnegative linearization of products at least for all $(\alpha,\beta)\in\Delta$ (cf. Section~\ref{sec:intro}), the converse ``(ii) $\Rightarrow$ (i)'' is a consequence of the assertion made in the second part of Theorem~\ref{thm:mainodd}.\\

In view of these observations, and in view of \eqref{eq:gTmixed} and \eqref{eq:gTadd}, Theorem~\ref{thm:mainodd} and Theorem~\ref{thm:mainfull} trace back to the following lemma:

\begin{lemma}\label{lma:centrallemma}
Let $(\alpha,\beta)\in V^{\prime}\backslash\Delta$, and let $m\in\mathbb{N}$, $s\in\mathbb{N}_0$. Then $g_T(2m+1,2m+2s+1;2s+2j)>0$ for all $j\in\{0,\ldots,2m+1\}$.
\end{lemma}

Due to the positivity of the sequences $(a_n^T)_{n\in\mathbb{N}}$ and $(c_n^T)_{n\in\mathbb{N}}$, the assertion of Lemma~\ref{lma:centrallemma} is also true for $m=0$, of course.\\

Our task is to establish Lemma~\ref{lma:centrallemma}, which will be done via Corollary~\ref{cor:oscillatory}, another ``two-sided induction'' argument (cf. the proof of Theorem~\ref{thm:gasperpositivevariant}) and an auxiliary result. The latter will be needed for the (more involved) induction step.\\

For the rest of the section, we always assume that $(\alpha,\beta)\in V^{\prime}\backslash\Delta$ and that $m\in\mathbb{N}$, $s\in\mathbb{N}_0$.\\

Under these conditions, we have
\begin{equation*}
a\in\left(-\frac{1}{3},0\right)
\end{equation*}
and
\begin{equation*}
b\in(-a,1+a)\subseteq(0,1).
\end{equation*}
The inequality $a>-1/3$ follows immediately from
\begin{equation*}
0\leq a^2+2b^2+3a<a^2+2(1+a)^2+3a=(a+2)(3a+1)
\end{equation*}
(and can also be found in \cite{Ga70b}). The inequality $b>-a$ was shown above.\\

We now define two auxiliary functions $p:\{1,\ldots,2m-1\}\rightarrow\mathbb{R}$, $q:\{1,\ldots,2m-1\}\rightarrow(0,\infty)$ by
\begin{align*}
p(j)&:=\frac{c_{2s+2j+3}^T}{a_{2s+2j+1}^T}\frac{\iota^+(m,m+s;j)}{\theta^+(m,m+s;j)},\\
q(j)&:=\frac{c_{2s+2j+1}^T c_{2s+2j+3}^T}{a_{2s+2j-1}^T a_{2s+2j+1}^T}\frac{\kappa^+(m,m+s;j)}{\theta^+(m,m+s;j)}.
\end{align*}
Concerning well-definedness, observe that $\theta^+(m,m+s;.)$ and $\kappa^+(m,m+s;.)$ are positive on their full domains, which follows directly from the definitions \eqref{eq:gRlambda}, \eqref{eq:gRnu}. Using \eqref{eq:gRlambda} to \eqref{eq:gRnu} and \eqref{eq:reccoeffgencheb}, we compute
\begin{equation}\label{eq:pqdecompose}
\begin{split}
p(j)&=p^{\infty}(j)+\frac{p^{\ast}(j)}{(2m-j+a)(2m+2s+j+a+2)},\\
q(j)&=q^{\infty}(j)+\frac{q^{\ast}(j)}{(2m-j+a)(2m+2s+j+a+2)}
\end{split}
\end{equation}
for all $j\in\{1,\ldots,2m-1\}$, where the four functions $p^{\infty}:\mathbb{N}\rightarrow(-1,\infty)$, $p^{\ast},q^{\infty},q^{\ast}:\mathbb{N}\rightarrow(0,\infty)$ are given by
\begin{equation}\label{eq:pqdecomposeexplicit}
\begin{split}
p^{\infty}(j)&=-1+\frac{2s+2j+a+2}{(2s+j+1)(2s+2j+a)(2s+2j+a+b+1)(j+1)}\\
&\quad\quad\quad\quad\times\left[b(2s+j+1)(2s+2j+a)(j+1)+(1-b)(2s+j)(2s+2j+a+1)j\right],\\
p^{\ast}(j)&=(1-b)\frac{(2s+j+a)(2s+2j+a+1)(2s+2j+a+2)(j+a)(2s+2j+1)}{(2s+j+1)(2s+2j+a)(2s+2j+a+b+1)(j+1)},\\
q^{\infty}(j)&=\frac{(2s+2j+a+2)(2s+j+a)(2s+2j+a-b+1)(j+a)}{(2s+j+1)(2s+2j+a)(2s+2j+a+b+1)(j+1)},\\
q^{\ast}(j)&=\frac{2s+2j+a+2}{(2s+j+1)(2s+2j+a)(2s+2j+a+b+1)(j+1)}\\
&\quad\times(1-a)(2s+j+a)(2s+2j+a+1)(j+a)(2s+2j+a-b+1).
\end{split}
\end{equation}
Note that the functions $p^{\infty}$, $p^{\ast}$, $q^{\infty}$ and $q^{\ast}$ are independent of $m$. The superscript ``$\infty$'' is used because $p^{\infty}$ and $q^{\infty}$ are just the pointwise limits of $p$ and $q$ as $m$ tends to infinity.\\

As a first consequence of \eqref{eq:pqdecompose} and \eqref{eq:pqdecomposeexplicit}, we obtain that $p$ maps into the interval $(-1,\infty)$.\\

The following lemma is the announced auxiliary result and provides an inequality in $p$ and $q$ which will be central in the proof of Lemma~\ref{lma:centrallemma}.

\begin{lemma}\label{lma:pqinequality}
Let $(\alpha,\beta)\in V^{\prime}\backslash\Delta$ and $m\geq2$, $s\in\mathbb{N}_0$. Then for every $j\in\{1,\ldots,2m-2\}$ the inequality
\begin{equation}\label{eq:pqinequality}
[1+p(j+1)][q(j)-p(j)]<q(j+1)
\end{equation}
is valid.
\end{lemma}

\begin{proof}
The basic idea is to use \eqref{eq:pqdecompose} and \eqref{eq:pqdecomposeexplicit} in order to isolate $m$ in an appropriate way. Let $j\in\{1,\ldots,2m-2\}$. We decompose
\begin{equation}\label{eq:factorm}
\begin{split}
&q(j+1)-[1+p(j+1)][q(j)-p(j)]=\\
&=q^{\infty}(j+1)-[1+p^{\infty}(j+1)][q^{\infty}(j)-p^{\infty}(j)]\\
&\quad+\frac{q^{\ast}(j+1)-p^{\ast}(j+1)[q^{\infty}(j)-p^{\infty}(j)]}{(2m-j+a-1)(2m+2s+j+a+3)}\\
&\quad-\frac{[1+p^{\infty}(j+1)][q^{\ast}(j)-p^{\ast}(j)]}{(2m-j+a)(2m+2s+j+a+2)}\\
&\quad-\frac{p^{\ast}(j+1)[q^{\ast}(j)-p^{\ast}(j)]}{(2m-j+a-1)(2m-j+a)(2m+2s+j+a+2)(2m+2s+j+a+3)}
\end{split}
\end{equation}
and compute
\begin{equation}\label{eq:omegaj}
\begin{split}
\omega_j:&=q^{\infty}(j+1)-[1+p^{\infty}(j+1)][q^{\infty}(j)-p^{\infty}(j)]=\\
&=\frac{(b-a)b[2s(2s+2j+a+2)+(j+a)(2j+4)+1-a]}{(2s+j+1)(2s+j+2)(j+1)(j+2)}\\
&\quad\times\frac{(2s+2j+a+2)(2s+2j+a+4)}{(2s+2j+a+b+1)(2s+2j+a+b+3)}>\\
&>0.
\end{split}
\end{equation}
Combining \eqref{eq:omegaj} with \eqref{eq:factorm}, we obtain
\begin{equation}\label{eq:factormmod}
\begin{split}
&\frac{(2m-j+a-1)(2m-j+a)(2m+2s+j+a+2)(2m+2s+j+a+3)}{\omega_j}\\
&\quad\times[q(j+1)-[1+p(j+1)][q(j)-p(j)]]=\\
&=[(2m-j+a-1)(2m+2s+j+a+3)+\alpha_j]\\
&\quad\times[(2m-j+a)(2m+2s+j+a+2)+\beta_j]+\rho_j=\\
&=[(2m-j+a-1)((2m-j+a-1)+\sigma_j+1)+\alpha_j]\\
&\quad\times[((2m-j+a-1)+1)((2m-j+a-1)+\sigma_j)+\beta_j]+\rho_j
\end{split}
\end{equation}
with
\begin{align*}
\alpha_j&:=\frac{q^{\ast}(j+1)-p^{\ast}(j+1)[q^{\infty}(j)-p^{\infty}(j)]}{\omega_j},\\
\beta_j&:=-\frac{[1+p^{\infty}(j+1)][q^{\ast}(j)-p^{\ast}(j)]}{\omega_j},\\
\rho_j&:=-\frac{p^{\ast}(j+1)[q^{\ast}(j)-p^{\ast}(j)]}{\omega_j}-\alpha_j\beta_j,\\
\sigma_j&:=2s+2j+3.
\end{align*}
We now define
\begin{equation*}
f:\left[\frac{j-a+1}{2},\infty\right)\rightarrow\mathbb{R}
\end{equation*}
by
\begin{equation*}
\begin{split}
f(x):&=[(2x-j+a-1)((2x-j+a-1)+\sigma_j+1)+\alpha_j]\\
&\quad\times[((2x-j+a-1)+1)((2x-j+a-1)+\sigma_j)+\beta_j]+\rho_j
\end{split}
\end{equation*}
and \textit{claim} that $f$ maps into $(0,\infty)$; once the claim is proven, inequality \eqref{eq:pqinequality} will follow for $j$ via
\begin{equation*}
m\in\left[\frac{j-a+1}{2},\infty\right),
\end{equation*}
\eqref{eq:omegaj} and \eqref{eq:factormmod}. To establish the claim, we first compute
\begin{equation*}
\begin{split}
f^{\prime}(x)&=[4(2x-j+a-1)+2\sigma_j+2]\\
&\quad\times\left[((2x-j+a-1)+1)((2x-j+a-1)+\sigma_j)+\beta_j\right.\\
&\quad\quad\left.+(2x-j+a-1)((2x-j+a-1)+\sigma_j+1)+\alpha_j\right].
\end{split}
\end{equation*}
Then, two rather tedious calculations yield
\begin{align*}
f\left(\frac{j-a+1}{2}\right)&=\alpha_j(\sigma_j+\beta_j)+\rho_j=\\
&=\frac{(2s+j+a+1)(2s+2j+3)(2s+2j+a+3)(j+a+1)}{b[2s(2s+2j+a+2)+(j+a)(2j+4)+1-a]}\\
&\quad\times[b(2s+j+1)(j+1)+(1-b)(2-a)(2s+2j+a+1)]>\\
&>0
\end{align*}
and, for each $x\geq(j-a+1)/2$,
\begin{align*}
&((2x-j+a-1)+1)((2x-j+a-1)+\sigma_j)+\beta_j\\
&\quad+(2x-j+a-1)((2x-j+a-1)+\sigma_j+1)+\alpha_j\geq\\
&\geq\sigma_j+\beta_j+\alpha_j=\\
&=\frac{1}{b[2s(2s+2j+a+2)+(j+a)(2j+4)+1-a]}\\
&\quad\times\left[b\left((1-a)(2s+j+2)(2s+j+a)(2s+2j+a+3)\right.\right.\\
&\quad\quad\quad\left.\left.+(1-a)(2s+2j+a+3)(j+1)(j+a+1)\right.\right.\\
&\quad\quad\quad\left.\left.+(2s+j+2)(2s+j+a+1)(j+a)(2j+4)\right.\right.\\
&\quad\quad\quad\left.\left.+2s(2s+2j+3)(2s+2j+a+2)+(j+1)(j+a)(2j+4)\right.\right.\\
&\quad\quad\quad\left.\left.+(1-a)(2s+2j+3)\right)\right.\\
&\quad\quad\left.+(1-b)\left((2s+j)(2j+a+2)+(2+a)j+2+3a\right)\right.\\
&\quad\quad\quad\left.\times(2s+2j+a+1)(2s+2j+a+3)\right]>\\
&>0
\end{align*}
(like in Section~\ref{sec:jacobi}, verifying such decompositions, which allow to directly see positivity, is easy, so the actual difficulty is finding them). Obviously, we also have
\begin{equation*}
4(2x-j+a-1)+2\sigma_j+2\geq2\sigma_j+2>0
\end{equation*}
for every $x\geq(j-a+1)/2$, so $f^{\prime}$ maps into $(0,\infty)$. This finishes the proof.
\end{proof}

We now come to the proof of Lemma~\ref{lma:centrallemma}.

\begin{proof}[Proof (Lemma~\ref{lma:centrallemma})]
As a consequence of Corollary~\ref{cor:oscillatory}, all numbers
\begin{equation*}
(-1)^j g_R^+(m,m+s;s+j),\;j\in\{0,\ldots,2m\},
\end{equation*}
are positive (observe that $(\beta+1,\alpha)$ is located in the interior of $\Delta$).\footnote{Alternatively, the positivity of the numbers $(-1)^j g_R^+(m,m+s;s+j)$ can be obtained from \eqref{eq:alphabetachange} and Rahman's formula \eqref{eq:Rahmanspecial} (take into account that $\beta+1>0>\alpha>-1/2$).} Hence, we may define $\phi:\{1,\ldots,2m\}\rightarrow(-\infty,0)$,
\begin{equation*}
\phi(j):=\frac{c_{2s+2j+1}^T}{a_{2s+2j-1}^T}\frac{g_R^+(m,m+s;s+j)}{g_R^+(m,m+s;s+j-1)}.
\end{equation*}
As a consequence of \eqref{eq:gRrecurrence}, we have
\begin{align*}
&p(j)+\frac{q(j)}{\phi(j)}=\\
&=\frac{c_{2s+2j+3}^T}{a_{2s+2j+1}^T}\frac{\iota^+(m,m+s;j)}{\theta^+(m,m+s;j)}\\
&\quad+\frac{c_{2s+2j+1}^T c_{2s+2j+3}^T}{a_{2s+2j-1}^T a_{2s+2j+1}^T}\frac{\kappa^+(m,m+s;j)}{\theta^+(m,m+s;j)}\frac{a_{2s+2j-1}^T}{c_{2s+2j+1}^T}\frac{g_R^+(m,m+s;s+j-1)}{g_R^+(m,m+s;s+j)}=\\
&=\frac{c_{2s+2j+3}^T}{a_{2s+2j+1}^T}\frac{g_R^+(m,m+s;s+j+1)}{g_R^+(m,m+s;s+j)}
\end{align*}
and obtain the recurrence relation
\begin{equation*}
\phi(j+1)=p(j)+\frac{q(j)}{\phi(j)}\;(1\leq j\leq2m-1).
\end{equation*}
We now use this recurrence relation and induction to show that
\begin{equation}\label{eq:phieven}
\phi(2j)<-1
\end{equation}
and
\begin{equation}\label{eq:phiodd}
\phi(2j-1)>-1
\end{equation}
for all $j\in\{1,\ldots,m\}$. As a consequence of \eqref{eq:necsecond} and
\begin{equation*}
a+2b>-a>0,
\end{equation*}
we see that $\phi(2)<-1$. Moreover, making use of \eqref{eq:gRrecurrence}, which yields
\begin{align*}
&\frac{g_R^+(m,m+s;s+2m-2)}{\underbrace{g_R^+(m,m+s;s+2m-1)}_{\neq0}}=\\
&=-\frac{\iota^+(m,m+s;2m-1)}{\underbrace{\kappa^+(m,m+s;2m-1)}_{>0}}+\frac{\theta^+(m,m+s;2m-1)}{\kappa^+(m,m+s;2m-1)}\frac{g_R^+(m,m+s;s+2m)}{g_R^+(m,m+s;s+2m-1)},
\end{align*}
and combining this with \eqref{eq:gRlambda} to \eqref{eq:gRnu}, \eqref{eq:gR2ms1} and \eqref{eq:reccoeffgencheb}, we obtain that
\begin{equation*}
\begin{split}
&4(b-1)\frac{(2m+a-1)(2m+s+a)(2m+2s+a-1)(4m+2s+a-b-1)}{4m+2s+a-2}\\
&\quad\times\left(\frac{a_{4m+2s-3}^T}{c_{4m+2s-1}^T}\frac{g_R^+(m,m+s;s+2m-2)}{g_R^+(m,m+s;s+2m-1)}+1\right)=\\
&=(2m+a-1)(2m+2s+a-1)\\
&\quad\quad\times\left[(a^2+2b^2+3a)(2m+s-1)-a(a+1)(2m+s-2)+(2+2a)b^2\right]\\
&\quad+(a+1)b(2-b)(4m+2s+a)(4m+2s+2a-1).
\end{split}
\end{equation*}
Therefore, we obtain that $\phi(2m-1)>-1$.\\

If $m=1$, then \eqref{eq:phieven} and \eqref{eq:phiodd} are already verified to hold for all $j\in\{1,\ldots,m\}$ by the preceding calculations; hence, we assume that $m\geq2$ from now on. Let $j\in\{1,\ldots,m-1\}$ be arbitrary but fixed and assume that $\phi(2j)<-1$. Then
\begin{equation*}
\phi(2j+1)=p(2j)+\frac{q(2j)}{\phi(2j)}>p(2j)-q(2j).
\end{equation*}
Since $p$ maps into $(-1,\infty)$, we obtain
\begin{equation*}
(1+p(2j+1))\phi(2j+1)>(1+p(2j+1))(p(2j)-q(2j)),
\end{equation*}
and now Lemma~\ref{lma:pqinequality} implies that
\begin{equation*}
(1+p(2j+1))\phi(2j+1)>-q(2j+1).
\end{equation*}
Since $\phi(2j+1)<0$, the latter equation yields
\begin{equation*}
\phi(2j+2)=p(2j+1)+\frac{q(2j+1)}{\phi(2j+1)}<-1.
\end{equation*}
Finally, let $j\in\{2,\ldots,m\}$ be arbitrary but fixed and assume that $\phi(2j-1)>-1$. We have
\begin{equation*}
\frac{1}{\phi(2j-2)}=\frac{1}{q(2j-2)}(\phi(2j-1)-p(2j-2))>-\frac{1}{q(2j-2)}(1+p(2j-2)),
\end{equation*}
so
\begin{equation*}
1+p(2j-2)>-\frac{q(2j-2)}{\phi(2j-2)}.
\end{equation*}
Since
\begin{equation*}
0>\phi(2j-2)=p(2j-3)+\frac{q(2j-3)}{\phi(2j-3)},
\end{equation*}
we can conclude that
\begin{equation*}
(1+p(2j-2))\left(p(2j-3)+\frac{q(2j-3)}{\phi(2j-3)}\right)<-q(2j-2).
\end{equation*}
Now we apply Lemma~\ref{lma:pqinequality} again and obtain
\begin{equation*}
(1+p(2j-2))\left(p(2j-3)+\frac{q(2j-3)}{\phi(2j-3)}\right)<(1+p(2j-2))(p(2j-3)-q(2j-3)).
\end{equation*}
Since $p$ maps into $(-1,\infty)$, this shows that
\begin{equation*}
p(2j-3)+\frac{q(2j-3)}{\phi(2j-3)}<p(2j-3)-q(2j-3)
\end{equation*}
or, equivalently, $\phi(2j-3)>-1$, which finishes the induction. Hence, \eqref{eq:phieven} and \eqref{eq:phiodd} are established to hold for all $j\in\{1,\ldots,m\}$ (for every $m\geq1$). Combining this with the positivity of all numbers $(-1)^j g_R^+(m,m+s;s+j)$ (see above) and \eqref{eq:gTodd}, we can conclude that all
\begin{align*}
&g_T(2m+1,2m+2s+1;2s+2j)=\\
&=a_{2s+2j-1}^T g_R^+(m,m+s;s+j-1)+c_{2s+2j+1}^T g_R^+(m,m+s;s+j)=\\
&=a_{2s+2j-1}^T\cdot\underbrace{(-1)^{j-1}g_R^+(m,m+s;s+j-1)}_{>0}\cdot\underbrace{(-1)^{j-1}(1+\phi(j))}_{>0},
\end{align*}
$j\in\{1,\ldots,2m\}$, are positive. Since the positivity of $g_T(2m+1,2m+2s+1;2s)$ and $g_T(2m+1,2m+2s+1;4m+2s+2)$ is clear, the proof is complete.
\end{proof}

\appendix

\section{Alternative proof of Lemma~\ref{lma:gRmu} via the mean value theorem}

As announced in Section~\ref{sec:jacobi}, we give an alternative proof of Lemma~\ref{lma:gRmu}. In fact, we establish the following stronger characterization result:

\begin{proposition}\label{prp:gRmu}
Let $\alpha,\beta>-1$, and let for $m\in\mathbb{N}$ and $s\in\mathbb{N}_0$ the function $\iota(m,m+s;.):[1,2m-1]\rightarrow\mathbb{R}$ be defined by \eqref{eq:gRmu}. If $b\neq0$ (i.e., if $\alpha\neq\beta$), then the following are equivalent:
\begin{enumerate}
\item[\rm(i)] For all $m\in\mathbb{N}$ and $s\in\mathbb{N}_0$, $\iota(m,m+s;.)$ has at most one zero.
\item[\rm(ii)] $a>-11/8+\sqrt{73}/8\approx-0.307$.
\end{enumerate}
Moreover, if $a\geq-11/8+\sqrt{73}/8$, then $\iota(m,m+s;1)\geq0$ for all $m\geq2,s\geq0$ if and only if $b\geq0$. Furthermore, if $a\geq0$, then $\iota(m,m+s;1)\geq0$ for all $m\geq1,s\geq0$ if and only if $b\geq0$.
\end{proposition}

Proposition~\ref{prp:gRmu} unifies and extends results which are contained in \cite{Ga70a,Ga70b}. Concerning some central parts contained in these references, we give a new proof which avoids Descartes' rule of signs and is based on the mean value theorem instead. Lemma~\ref{lma:gRmu} is indeed implied by Proposition~\ref{prp:gRmu} because $(\alpha,\beta)\in V$ implies $a>-11/8+\sqrt{73}/8$ \cite{Ga70b}; for the sake of completeness, we give the following short proof of the latter implication: let $(\alpha,\beta)\in V$. If $a\geq7/2+\sqrt{145}/2$, then the inequality $a>-11/8+\sqrt{73}/8$ is trivially true, so assume that $a<7/2+\sqrt{145}/2$ in the following. Since $7/2-\sqrt{145}/2<-1$, $(\alpha,\beta)$ satisfies the estimation
\begin{align*}
0&<\frac{1}{(a+3)(a+5)}\left[\left(\frac{7}{2}+\frac{1}{2}\sqrt{145}\right)-a\right]\left[a-\left(\frac{7}{2}-\frac{1}{2}\sqrt{145}\right)\right]=\\
&=2-3\frac{(a+1)(a+2)}{(a+3)(a+5)}
\end{align*}
and, consequently, the estimation
\begin{align*}
0&\leq a^2+\left[2-3\frac{(a+1)(a+2)}{(a+3)(a+5)}\right]b^2+3a<\\
&<a^2+\left[2-3\frac{(a+1)(a+2)}{(a+3)(a+5)}\right](a+1)^2+3a=\\
&=4\frac{(a+2)}{(a+3)(a+5)}(4a^2+11a+3)=\\
&=16\frac{(a+2)}{(a+3)(a+5)}\left[a-\left(-\frac{11}{8}+\frac{1}{8}\sqrt{73}\right)\right]\left[a-\left(-\frac{11}{8}-\frac{1}{8}\sqrt{73}\right)\right].
\end{align*}
Since $-11/8-\sqrt{73}/8<-1$, we obtain that $a>-11/8+\sqrt{73}/8$.

\begin{proof}[Proof (Proposition~\ref{prp:gRmu})]
We divide our proof into three steps. Step 1 and Step 2 deal with the equivalence ``(i) $\Leftrightarrow$ (ii)'', and Step 3 deals with the second part of the proposition. The most interesting part is the direction ``(ii) $\Rightarrow$ (i)'', which was obtained via Descartes' rule of signs in \cite{Ga70b}. As announced above, we present an alternative approach which just uses the mean value theorem.\\

\textit{Step 1:} let $b\neq0$, $m\in\mathbb{N}$ and $s\in\mathbb{N}_0$. Furthermore, in this step we additionally assume that $s\neq0$ or $a\geq0$, and we write
\begin{equation*}
\iota(m,m+s;j)=b[f(j+1)-f(j)]
\end{equation*}
with $f:[1,2m]\rightarrow(0,\infty)$,
\begin{equation*}
f(j):=(2m-j+1)(2m+2s+j+2a-1)\frac{2s+j}{2s+2j+a-1}j.
\end{equation*}
Our aim is to show that the function $\iota(m,m+s;.)$ has at most one zero. If $\iota(m,m+s;.)$ had two different zeros $j_1,j_2\in[1,2m-1]$, $j_1<j_2$, then
\begin{equation}\label{eq:meanvalueI}
0=f(j_1+1)-f(j_1)=f(j_2+1)-f(j_2)
\end{equation}
and therefore
\begin{equation}\label{eq:meanvalueII}
f(j_2+1)-f(j_1+1)=f(j_2)-f(j_1).
\end{equation}
In the following, we want to conclude that there are $j_1^{\prime},j_2^{\prime}\in(1,2m)$ with $j_1^{\prime}<j_2^{\prime}$ and $0=f^{\prime}(j_1^{\prime})=f^{\prime}(j_2^{\prime})$ (which will yield a contradiction). We distinguish two cases.
\begin{itemize}
\item On the one hand, if $j_2\geq j_1+1$, then \eqref{eq:meanvalueI} and the mean value theorem immediately yield the existence of $j_1^{\prime}\in(j_1,j_1+1)$ and $j_2^{\prime}\in(j_2,j_2+1)$ with $0=f^{\prime}(j_1^{\prime})=f^{\prime}(j_2^{\prime})$.
\item On the other hand, if $j_2<j_1+1$, then \eqref{eq:meanvalueII} and the mean value theorem yield $j_1^{\prime\prime}\in(j_1,j_2)$ and $j_2^{\prime\prime}\in(j_1+1,j_2+1)$ such that
\begin{equation*}
f^{\prime}(j_1^{\prime\prime})=f^{\prime}(j_2^{\prime\prime})=\frac{f(j_2)-f(j_1)}{j_2-j_1}.
\end{equation*}
Now if $f(j_2)=f(j_1)$, then we have $0=f^{\prime}(j_1^{\prime})=f^{\prime}(j_2^{\prime})$ for $j_1^{\prime}:=j_1^{\prime\prime}$ and $j_2^{\prime}:=j_2^{\prime\prime}$. If, however, $f(j_2)>f(j_1)$, then $f^{\prime}(j_1^{\prime\prime})=f^{\prime}(j_2^{\prime\prime})>0$, and another application of the mean value theorem yields some $j_3^{\prime\prime}\in(j_2,j_1+1)$ with
\begin{equation*}
f^{\prime}(j_3^{\prime\prime})=\frac{f(j_1+1)-f(j_2)}{j_1+1-j_2}=\frac{f(j_1)-f(j_2)}{j_1+1-j_2}<0,
\end{equation*}
so we can find $j_1^{\prime}\in(j_1^{\prime\prime},j_3^{\prime\prime})$, $j_2^{\prime}\in(j_3^{\prime\prime},j_2^{\prime\prime})$ such that $0=f^{\prime}(j_1^{\prime})=f^{\prime}(j_2^{\prime})$. Finally, if $f(j_2)<f(j_1)$, we can conclude in an analogous way.
\end{itemize}
Hence, in any case there are $j_1^{\prime},j_2^{\prime}\in(1,2m)$ with $j_1^{\prime}<j_2^{\prime}$ and $0=f^{\prime}(j_1^{\prime})=f^{\prime}(j_2^{\prime})$. To obtain a contradiction, we now decompose $f=u v$ with $u,v:[1,2m]\rightarrow(0,\infty)$,
\begin{align*}
u(j)&:=(2m+2s+j+2a-1)(2s+j)j,\\
v(j)&:=\frac{2m-j+1}{2s+2j+a-1}.
\end{align*}
For any $j\in[1,2m]$, one has $f^{\prime}(j)=0$ if and only if $u^{\prime}(j)/u(j)+v^{\prime}(j)/v(j)=0$, or equivalently
\begin{equation}\label{eq:gRplusmumonotony}
\begin{split}
\frac{1}{2m-j+1}&=\frac{1}{2m+2s+j+2a-1}+\frac{1}{2s+j}+\frac{1}{j}-\frac{2}{2s+2j+a-1}=\\
&=\frac{1}{2m+2s+j+2a-1}+\frac{1}{2s+j}+\frac{2s+a-1}{j(2s+2j+a-1)}.
\end{split}
\end{equation}
At this stage, we distinguish two cases again.\\

\textit{Case 1:} $s\neq0$. Since $2s+a-1>0$, we see that the right-hand side of equation \eqref{eq:gRplusmumonotony} is strictly decreasing w.r.t. $j\in[1,2m]$, whereas the left-hand side is strictly increasing. We thus obtain that $f^{\prime}$ can have at last one zero, a contradiction.\\

\textit{Case 2:} $s=0$. Then, by the additional assumption, $a\geq0$. Moreover, \eqref{eq:gRplusmumonotony} reduces to
\begin{equation*}
\frac{1}{2m-j+1}=\frac{1}{2m+j+2a-1}+\frac{1}{j}+\frac{a-1}{j(2j+a-1)}
\end{equation*}
or, equivalently,
\begin{equation*}
\frac{2j+2a-2}{(2m-j+1)(2m+j+2a-1)}=\frac{2j+2a-2}{j(2j+a-1)}.
\end{equation*}
Since $a\geq0$, every zero of $f^{\prime}|_{(1,2m]}$ must therefore satisfy
\begin{equation*}
(2m-j+1)(2m+j+2a-1)=j(2j+a-1).
\end{equation*}
We now define $\eta:[1,2m]\rightarrow\mathbb{R}$,
\begin{equation*}
\eta(j):=j(2j+a-1)-(2m-j+1)(2m+j+2a-1).
\end{equation*}
Since
\begin{equation*}
\eta^{\prime}(j)=6j+3a-3>0
\end{equation*}
for all $j\in[1,2m]$, the function $\eta$ is strictly increasing and we obtain that $f^{\prime}|_{(1,2m]}$ can have at last one zero, which yields a contradiction again.\\

Hence, if ($b\neq0$, $m\in\mathbb{N}$, $s\in\mathbb{N}_0$ and) at least one of the additional conditions $s\neq0$ or $a\geq0$ holds, then $\iota(m,m+s;.)$ has at most one zero.\\

\textit{Step 2:} let $b\neq0$ again. For every $m\in\mathbb{N}$ we rewrite
\begin{equation*}
\iota(m,m;j)=-\frac{b}{(2j+a-1)(2j+a+1)}\chi_m(j)
\end{equation*}
with $\chi_m:[1,2m-1]\rightarrow\mathbb{R}$,
\begin{align*}
\chi_m(j):&=(2m-j+1)(2m+j+2a-1)j^2(2j+a+1)\\
&\quad-(2m-j)(2m+j+2a)(j+1)^2(2j+a-1)=\\
&=(4m+a+1)(2m+j+2a-1)j^2\\
&\quad-(2m-j)(2j+a-1)[(2m+j+2a)(2j+1)+j^2].
\end{align*}
In view of Step 1, concerning the assertion ``(i) $\Leftrightarrow$ (ii)'' it is left to establish the directions ``(i') $\Rightarrow$ (ii)'' and ``(ii') $\Rightarrow$ (i')'' with
\begin{enumerate}
\item[\rm(i')] For all $m\in\mathbb{N}$, $\chi_m$ has at most one zero in $[1,2m-1]$.
\item[\rm(ii')] $a\in(-11/8+\sqrt{73}/8,0)$.
\end{enumerate}
We consider $\chi_m$ on the whole real line (by canonical extension) and conclude as follows:\\

``(i') $\Rightarrow$ (ii)'': if $a\leq-11/8+\sqrt{73}/8$, then
\begin{align*}
\chi_2(1)&=-16a^2-44a-12=\\
&=-16\left[a-\left(-\frac{11}{8}+\frac{1}{8}\sqrt{73}\right)\right]\left[a-\left(-\frac{11}{8}-\frac{1}{8}\sqrt{73}\right)\right]\geq\\
&\geq0,\\
\chi_2(2)&=-12(a+1)(a+2)<0,\\
\chi_2(3)&=4a^2+88a+196>0.
\end{align*}
Consequently, $\chi_2$ has both a zero in $[1,2)$ and a zero in $(2,3)$, which violates (i').\\

``(ii') $\Rightarrow$ (i')'': the case $m=1$ is trivial, so let $m\geq2$. Then the estimation $a\in(-1/3,0)$ implies
\begin{align*}
\chi_m(-1)&=(4a-4)(m+1)(m+a-1)<0,\\
\chi_m(0)&=(4-4a)m(m+a)>0,\\
\chi_m(1)&=-(4+12a)m(m+a)+4(a+1)(2a+1)\leq\\
&\leq-(4+12a)\cdot2(2+a)+4(a+1)(2a+1)=\\
&=\chi_2(1)=\\
&=-16\left[a-\left(-\frac{11}{8}+\frac{1}{8}\sqrt{73}\right)\right]\left[a-\left(-\frac{11}{8}-\frac{1}{8}\sqrt{73}\right)\right]<\\
&<0.
\end{align*}
Hence, together with $\lim_{j\to-\infty}\chi_m(j)=\infty$ we obtain that $\chi_m$ has a zero in $(-\infty,-1)$, that $\chi_m$ has a zero in $(-1,0)$ and that $\chi_m$ has also a zero in $(0,1)$. As a polynomial in $j$ of degree four, however, this implies that $\chi_m$ can have at most one zero in $[1,2m-1]$.\\

\textit{Step 3:} the second part of the proposition is a consequence of the representations
\begin{align*}
&\iota(m,m+s;1)=\\
&=\frac{4b}{(2s+a+1)(2s+a+3)}\\
&\quad\times\left[((2s+3)(2s+a+1)-(4s+2))(m-1)(m+s+a+1)+a(a+1)\right]=\\
&=\frac{4b}{(2s+a+1)(2s+a+3)}\\
&\quad\times\left[((2s+3)(2s+a+1)-(4s+2))(m-2)(m+s+a+2)\right.\\
&\quad\quad\left.+(2(s+a+2)(2s+a+2)+4s+5a+5)s+4a^2+11a+3\right],
\end{align*}
the estimation
\begin{equation*}
(2s+3)(2s+a+1)-(4s+2)>(2s+3)\left(2s-\frac{1}{3}+1\right)-(4s+2)=\frac{2}{3}s(6s+5)\geq0
\end{equation*}
(which is satisfied for all $a\geq-11/8+\sqrt{73}/8>-1/3$) and the factorization
\begin{equation*}
4a^2+11a+3=4\left[a-\left(-\frac{11}{8}+\frac{1}{8}\sqrt{73}\right)\right]\left[a-\left(-\frac{11}{8}-\frac{1}{8}\sqrt{73}\right)\right].
\end{equation*}
\end{proof}

\section{Correction of Rahman's hypergeometric representations}

This short section contains the announced corrections of small mistakes in Rahman's hypergeometric representations \cite[(1.7) to (1.9)]{Ra81a} of the linearization coefficients $g_R(m,n;k)$ (which belong to the Jacobi polynomials $(R_n^{(\alpha,\beta)}(x))_{n\in\mathbb{N}_0}$). For all $m\in\mathbb{N}$ and $s\in\mathbb{N}_0$, one has
\begin{equation}\label{eq:Rahmaneven}
\begin{split}
&g_R(m,m+s;s+j)=\\
&=\frac{\alpha+\beta+1+2s+2j}{\alpha+\beta+1}(m+\alpha+\beta+1)_m\\
&\quad\times\frac{(\alpha+1)_{s+j}(\beta+1)_{m+s}(\alpha+\beta+1)_{2s+j}(\alpha+\beta+1)_j(m+s)!}{(\alpha+1)_s(\alpha+1)_m(\beta+1)_{s+j}(\alpha+\beta+2)_{2m+2s+j}s!j!}\\
&\quad\times\frac{(-m)_{\frac{j}{2}}(\alpha+\beta+m+s+1)_{\frac{j}{2}}}{\left(-m-\frac{\alpha+\beta}{2}\right)_{\frac{j}{2}}(\alpha+s+1)_{\frac{j}{2}}}\\
&\quad\times\frac{(-m-\alpha)_{\frac{j}{2}}(\beta+m+s+1)_{\frac{j}{2}}\left(\frac{1}{2}\right)_{\frac{j}{2}}}{\left(\frac{1}{2}-m-\frac{\alpha+\beta}{2}\right)_{\frac{j}{2}}(s+1)_{\frac{j}{2}}(\alpha+1)_{\frac{j}{2}}}\\
&\quad\times{}_9F_8\left(\begin{matrix}\alpha,1+\frac{\alpha}{2},\alpha+\frac{1}{2},\frac{\alpha-\beta}{2},\frac{\alpha-\beta+1}{2},\alpha+\beta+m+s+1+\frac{j}{2}, \\ \frac{\alpha}{2},\frac{1}{2},\frac{\alpha+\beta}{2}+1,\frac{\alpha+\beta+1}{2},-\beta-m-s-\frac{j}{2},\end{matrix}\right.\\
&\quad\quad\quad\quad\quad\left.\left.\begin{matrix}-m+\frac{j}{2},-s-\frac{j}{2},-\frac{j}{2} \\ \alpha+m+1-\frac{j}{2},\alpha+s+1+\frac{j}{2},\alpha+1+\frac{j}{2}\end{matrix}\right|1\right)
\end{split}
\end{equation}
for even $j\in\{0,\ldots,2m\}$ and
\begin{equation}\label{eq:Rahmanodd}
\begin{split}
&g_R(m,m+s;s+j)=\\
&=\frac{\alpha+\beta+1+2s+2j}{\alpha+\beta+1}(m+\alpha+\beta+1)_m\\
&\quad\times\frac{(\alpha+1)_{s+j}(\beta+1)_{m+s}(\alpha+\beta+1)_{2s+j}(\alpha+\beta+1)_j(m+s)!}{(\alpha+1)_s(\alpha+1)_m(\beta+1)_{s+j}(\alpha+\beta+2)_{2m+2s+j}s!j!}\\
&\quad\times\frac{(-m)_{\frac{j+1}{2}}(\alpha+\beta+m+s+1)_{\frac{j+1}{2}}}{\left(-m-\frac{\alpha+\beta}{2}\right)_{\frac{j+1}{2}}(\alpha+s+1)_{\frac{j+1}{2}}}\\
&\quad\times\frac{(-m-\alpha)_{\frac{j-1}{2}}(\beta+m+s+1)_{\frac{j-1}{2}}\left(\frac{3}{2}\right)_{\frac{j-1}{2}}}{\left(\frac{1}{2}-m-\frac{\alpha+\beta}{2}\right)_{\frac{j-1}{2}}(s+1)_{\frac{j-1}{2}}(\alpha+2)_{\frac{j-1}{2}}}\\
&\quad\times\frac{\alpha-\beta}{\alpha+\beta+1}{}_9F_8\left(\begin{matrix}\alpha+1,\frac{\alpha+3}{2},\alpha+\frac{1}{2},\frac{\alpha-\beta}{2}+1,\frac{\alpha-\beta+1}{2}, \\ \frac{\alpha+1}{2},\frac{3}{2},\frac{\alpha+\beta}{2}+1,\frac{\alpha+\beta+3}{2},\end{matrix}\right.\\
&\quad\quad\quad\quad\quad\quad\quad\quad\quad\left.\left.\begin{matrix}\alpha+\beta+m+s+\frac{3}{2}+\frac{j}{2},-m+\frac{1}{2}+\frac{j}{2},\frac{1}{2}-s-\frac{j}{2},\frac{1-j}{2} \\ \frac{1-j}{2}-\beta-m-s,\alpha+m+\frac{3}{2}-\frac{j}{2},\alpha+s+\frac{3}{2}+\frac{j}{2},\alpha+\frac{3}{2}+\frac{j}{2}\end{matrix}\right|1\right)
\end{split}
\end{equation}
for odd $j\in\{0,\ldots,2m\}$, which corrects \cite[(1.7), (1.8)]{Ra81a}. This shows the nonnegativity of the $g_R(m,n;k)$ for $(\alpha,\beta)\in\Delta$, as well as the strict positivity for $(\alpha,\beta)\in\Delta^{\circ}$. For the subcase $\alpha\geq\beta\geq-1/2$, the nonnegativity of the $g_R(m,n;k)$ can also seen via the representation
\begin{equation}\label{eq:Rahmanspecial}
\begin{split}
g_R(m,m+s;s+j)&=\frac{\alpha+\beta+1+2s+2j}{\alpha+\beta+1}\cdot\frac{(m+s)!}{s!j!}\\
&\quad\times\frac{(\beta+1)_{m+s}(\alpha+\beta+1)_{2m}}{(\alpha+1)_m(\beta+1)_s(\alpha+\beta+1)_m}\\
&\quad\times\frac{(\alpha+\beta+1)_{2s+j}(-2m)_j(2\alpha+2\beta+2m+2s+2)_j}{(\alpha+\beta+2)_{2m+2s+j}(-2m-\alpha-\beta)_j}\\
&\quad\times\frac{(\alpha-\beta)_j}{(2\beta+2s+2)_j}\\
&\quad\times{}_9F_8\left(\begin{matrix}\beta+s+\frac{1}{2},1+\frac{\beta+s+\frac{1}{2}}{2},\beta+\frac{1}{2},\beta+m+s+1,-m-\alpha, \\ \frac{\beta+s+\frac{1}{2}}{2},s+1,-m+\frac{1}{2},\alpha+\beta+m+s+\frac{3}{2},\end{matrix}\right.\\
&\quad\quad\quad\quad\quad\left.\left.\begin{matrix}\frac{\alpha+\beta+1}{2}+s+\frac{j}{2},\frac{\alpha+\beta+2}{2}+s+\frac{j}{2},\frac{1-j}{2},-\frac{j}{2} \\ \frac{\beta-\alpha}{2}+\frac{2-j}{2},\frac{\beta-\alpha}{2}+\frac{1-j}{2},\beta+s+1+\frac{j}{2},\beta+s+\frac{3}{2}+\frac{j}{2}\end{matrix}\right|1\right),
\end{split}
\end{equation}
which is valid for all $m\in\mathbb{N}$, $s\in\mathbb{N}_0$ and $j\in\{0,\ldots,2m\}$ and which corrects a typo in \cite[(1.9)]{Ra81a}. Note that the expressions in \eqref{eq:Rahmaneven} to \eqref{eq:Rahmanspecial} may not be well-defined if $(\alpha,\beta)$ is an element of the boundary of $\Delta$; in this case, the formulas have to be interpreted as limits.

\bibliography{bibliographyjacobigencheb}
\bibliographystyle{amsplain}

\end{document}